%% file: DiDe_TermSets_v10.tex
\documentclass[10pt, a4paper, twoside,reqno]{amsart}

\usepackage{verbatim}

\usepackage{fancyhdr}
\usepackage{mathabx}

\makeatletter

\def\section{\@startsection{section}{1}%
	\z@{.7\linespacing\@plus\linespacing}{.5\linespacing}%
	{\normalfont \Large\scshape\centering}}

\def\subsection{\@startsection{subsection}{2}%
	\z@{.5\linespacing\@plus.7\linespacing}{.5\linespacing}%
	{\normalfont\large\bfseries}}

\def\subsubsection{\@startsection{subsubsection}{3}%
	\z@{.5\linespacing\@plus.7\linespacing}{.5\linespacing}%
	{\normalfont\itshape}}

\usepackage[usenames, dvipsnames]{color}
\definecolor{darkblue}{rgb}{0.0, 0.0, 0.45}

\usepackage[colorlinks	= true,
raiselinks	= true,
linkcolor	= darkblue, 
citecolor	= Mahogany,
urlcolor	= ForestGreen,
pdfauthor	= {Georgios Darivianakis},
pdftitle	= {},
pdfkeywords	= {},
pdfsubject	= {},
plainpages	= false]{hyperref}

\allowdisplaybreaks
\date{\today}

\usepackage{graphicx, psfrag}
\usepackage{amsthm}
\usepackage{amsmath} 
\usepackage{amssymb}  
\usepackage[noadjust]{cite}
\usepackage{color}
\usepackage{enumerate}
\usepackage{subfigure}
\usepackage{multirow}
\usepackage{booktabs}
\usepackage{subfigure}
\graphicspath{{drawing/}}

\usepackage{tikz}
\usetikzlibrary{calc,patterns,decorations.pathmorphing,decorations.markings}

\usepackage{xspace}

\newtheorem{definition}{Definition}
\newtheorem{proposition}{Proposition}

\newtheorem{theorem}{Theorem}

\usepackage{tikz}
\usetikzlibrary{calc,patterns,decorations.pathmorphing,decorations.markings}

\newcommand{\mb}{\mathbb}
\newcommand{\mc}{\mathcal}

\newcommand{\st}{\textrm{s.t.}}

\newcommand{\Ni}{{\scriptscriptstyle \mathcal N_i \scriptstyle}}
\newcommand{\Ti}{{\scriptscriptstyle T \scriptstyle}}
\newcommand{\Fi}{{\scriptscriptstyle f \scriptstyle}}

\newcommand{\wh}[1]{\widehat{#1}}
\newcommand{\wt}[1]{\widetilde{#1}}

\graphicspath{{./figs/}}

\title[Distributed Model Predictive Control for Linear Systems with Adaptive Terminal Sets]{Distributed Model Predictive Control for Linear Systems with Adaptive Terminal Sets}

\author{G. Darivianakis, A. Eichler and J. Lygeros
}%
\thanks{This research project is financially supported by the Swiss Innovation Agency Innosuisse and is part of the Swiss Competence Center for Energy Research SCCER FEEB\&D} 
\thanks{The authors are with the Automatic Control Laboratory, Department of Electrical Engineering and Information Technology, ETH Zurich, 8092 Zurich, Switzerland. (e-mail: gdarivia@control.ee.ethz.ch; eichlean@control.ee.ethz.ch; jlygeros@control.ee.ethz.ch).}

\begin{document} 

\begin{abstract}
	In this paper, we propose a distributed model predictive control (DMPC) scheme for linear time-invariant constrained systems which admit a separable structure. To exploit the merits of distributed computation algorithms, the stabilizing terminal controller, value function and invariant terminal set of the DMPC optimization problem need to respect the loosely coupled structure of the system. Although existing methods in the literature address this task, they typically decouple the synthesis of terminal controllers and value functions from the one of terminal sets. In addition, these approaches do not explicitly consider the effect of the current state of the system in the synthesis process. These limitations can lead the resulting DMPC scheme to poor performance since it may admit small or even empty terminal sets. Unlike other approaches, this paper presents a unified framework to encapsulate the synthesis of both the stabilizing terminal controller and invariant terminal set into the DMPC formulation. Conditions for Lyapunov stability and invariance are imposed in the synthesis problem in a way that allows the value function and invariant terminal set to admit the desired distributed structure. We illustrate the effectiveness of the proposed method on several examples including a benchmark spring-mass-damper problem.	
\end{abstract}

\maketitle

\section{Introduction}

Operation of large-scale networks of interacting dynamical systems remains an active field of research due to its high impact on real-world applications, e.g., regulation of power networks \cite{Venkat2008} and energy management of building districts \cite{Darivianakis2016}. For a system of this scale, the design and deployment of a centralized controller to regulate its operation is often a difficult task due to computation and communication limitations in the network. In such cases, it is desirable to design interacting local controllers with a prescribed structure which rely only on local information and computational resources. Even though the problem of synthesizing optimal distributed controllers is known NP-hard \cite{Tsitsiklis1985} in its general form, for certain network structures it has been shown to admit either a closed-form solution \cite{Fardad2009} or an exact convex reformulation \cite{Rotkowitz2006}. For general network structures, the usual practice is to resort to a linear matrix inequality (LMI) relaxations \cite{Langbort2004,Zecevic2010,Eichler2014} or semidefinite programming (SDP) relaxations \cite{Lavaei2012,Fazelnia2017} to obtain suboptimal distributed controllers with performance guarantees.

A downside of these static distributed controllers is their inability to efficiently cope with state and input constraints of the systems. Model predictive control (MPC) is an optimization based methodology that is well-suited for constrained linear systems \cite{Mayne2000}. Despite recent advances on computation and communication technologies, formulating and solving a large optimization problem, within the existing time limitations, remains a challenging task. To circumvent this, several methods have been proposed in the literature to leverage the distributed structure of the network in order to approximate the original optimization problem through a set of loosely coupled subproblems. DMPC approaches are typically categorized into non-cooperative \cite{Mayne2005,Richards2007,Trodden2010,Farina2012,Riverso2014,Lucia2015,Trodden2017} and cooperative ones \cite{Keviczky2004,Venkat2005,Stewart2010,Maestre2011,Giselsson2014,Conte2012,DarivianakisACC}. In the former, each system considers the effect of neighboring systems as a disturbance in its own dynamics and constraints. Exemplary cases of non-cooperative DMPC approaches are tube-based methods where the states and inputs of neighboring subsystems are confined in a precomputed \cite{Mayne2005,Richards2007,Trodden2010} or adaptive \cite{Farina2012,Riverso2014,Lucia2015,Trodden2017} bounded set. In this setting, each subsystem needs to account for all possible impacts of its neighbors occurring within these bounded sets. Though computationally simple and effective in practice, non-cooperative approaches can be conservative in presence of strong coupling. On the other hand, cooperative distributed MPC approaches require substantial communication infrastructure and computation resources since a system-wide MPC problem is formulated and solved. Approaches discussed in the literature \cite{Keviczky2004,Venkat2005,Stewart2010,Maestre2011,Giselsson2014} typically involve the communication of planned control sequences or state trajectories between neighboring systems. Unlike the conservative non-cooperative methods, cooperative approaches can guarantee convergence to the optimal solution of the original centralized optimization problem.

In the MPC scheme, the existence of a stabilizing static terminal controller is needed to guarantee recursive feasibility and stability of the closed-loop system. This terminal controller respects the state and input constraints of the system when operated in an invariant terminal set. The infinite-horizon cost associated with this terminal controller is upper bounded by a value function \cite{Mayne2000}. In the DMPC framework adopted here, the terminal controller, value function and invariant terminal set are designed as to respect the existing distributed structure of the system \cite{Mayne2005,Richards2007,Trodden2010,Farina2012,Riverso2014,Lucia2015,Trodden2017,Keviczky2004,Venkat2005,Stewart2010,Maestre2011,Giselsson2014,Conte2012,DarivianakisACC}. This way, the resulting DMPC optimization problem admits the desired distributed structure that makes it amendable to distributed computation algorithms such as the alternating direction method of multipliers (ADDM) \cite{Boyd2011}. To achieve this, current approaches in the literature typically split the design phase into two sequential parts: $ (i) $ the terminal controller and value function are synthesized based on Lyapunov stability concepts, then $ (ii) $ the invariant terminal set is constructed as the closed-loop system under the given terminal controller to satisfy the state and input constraints of the system. However, the resulting invariant terminal set can be a small (or even empty) inner approximation of the maximum invariant terminal set due to the imposed restrictions on its structure and the decoupled design phases. This can lead to severe performance degradation of the resulting DMPC scheme.

In this paper, we propose a novel approach which allows us to encapsulate the design of a distributed stabilizing terminal controller and invariant terminal set in the DMPC formulation such that these can be adapted in every iteration given the current state of the system. The necessity of online adaptation of terminal sets based on the predicted system evolution has previously been identified on several works (e.g., \cite{Conte2012, Simon2014, Lucia2015, Trodden2017}). The key difference of our approach is that the design of the stabilizing terminal controller and the invariant terminal set as well as the derivation of the optimal input for the DMPC problem are the result of one single optimization problem. This is beneficial since the size of the invariant terminal sets is now determined together with the predicted system evolution and explicitly depends on the current state of the system. This way the conservativeness introduced by imposing a decentralized structure on the invariant terminal sets is reduced by allowing flexibility on the shape of these sets. For the the design of decoupled terminal invariant sets the mutual dependencies of the neighboring systems are treated as a bounded disturbance. We employ robust optimization tools to express the Lyapunov stability and invariance conditions explicitly on the DMPC optimization problem in the form of LMIs. These LMIs are formulated as to respect the existing coupling structure of the system. Although mutual dependencies are treated as disturbances, the proposed method falls in the category of cooperative schemes since the sizes of the invariant terminal sets as well as the input trajectories for the finite-time MPC horizon are optimization variables that need to be agreed among all the involved systems in the network.

The rest of the paper is organized as follows. In Section \ref{sec::probForm} the dynamical system is analyzed and the standard DMPC scheme is briefly reviewed. The main contributions are presented in Section \ref{sec::design}, where the methods to encapsulate the design of the distributed stabilizing terminal controller, value function and invariant terminal set, based on Lyapunov stability and invariance conditions, in the DMPC problem formulation are discussed. Section \ref{sec::numericals} provides numerical studies to assess the efficacy and scalability of the proposed method. Concluding remarks are provided in Section \ref{sec::conclusion}.

{\bf Notation:} Let $\mathbb{R}$, $ \mb R_+ $ and $\mathbb{N}_+ $ denote the set of real numbers, non-negative real numbers and non-negative integers, respectively. For a vector $ v \in \mb R^n $, we denote by $v^\top$ its transpose and $ \| v \| $ its Euclidean norm. For given vectors $ v_{i} \in \mb R^{k_i} $ with $ k_i \in \mb N $, $ i \in \mc M = \{1,\ldots,m\} $, we define $ [v_{i}]_{i\in \mc M} = [v_{1}^\top \ldots v_{m}^\top]^\top \in \mb R^{k} $ with $ k = \sum_{i=1}^{m}k_i $ as their vector concatenation, and $ \text{diag}(v_1,\ldots,v_M) $ as the block diagonal matrix with $ v_1, \ldots ,v_M $ on the diagonal and zeros elsewhere. The notation $W\succeq 0$ is used to show that a symmetric matrix $W$ is positive semidefinite. A function $ f: \mb R_+ \rightarrow \mb R_+ $ belongs to class $ \mc K $ if it is continuous, strictly increasing and $ f(0) = 0 $. A function $ f: \mb R_+ \rightarrow \mb R_+ $ belongs to class $ \mc K_\infty $ if $ f \in \mc K $ and $ \lim_{x \rightarrow \infty} f(x) = \infty $.


\section{Problem formulation}\label{sec::probForm}

\subsection{Dynamically coupled constrained linear systems}
\begin{subequations}\label{eq::system}
	Consider a discrete-time linear time-invariant system with state dynamics at time $ t \in \mb N_+ $ given as
	\begin{equation}
	x_{t+1} = A x_t + B u_t.
	\end{equation}
	Here, $ x_t \in \mb R^n $ denotes the states with $ x_0 $ known and $ u_t \in \mb R^m $ the control inputs. The system matrices $ A \in \mb R^{n \times n} $, $ B \in \mb R^{n \times m} $ are known with $ (A, B) $ to form a controllable pair. The states and inputs of the system are subject to linear constraints
	\begin{align}
	& x_t \in \mc X = \{ x \in \mb R^{n} \; : \; G x \le g \},\\
	& u_t \in \mc U = \{ u \in \mb R^{m} \; : \; H u \le h \},
	\end{align}
	with known matrices $ G \in \mb R^{p \times n} $, $ g \in \mb R^p $, $ H \in \mb R^{k \times m} $ and $ h \in \mb R^m $. To simplify exposition, we assume that the sets $ \mc X $ and $ \mc U $ contain the origin in their interior. The optimal control law is defined through the optimizer that minimizes the infinite-horizon objective function
	\begin{equation}\label{eq::ObjFnc}
	J_{\infty} = \sum_{t=0}^{\infty} \ell(x_{t},u_{t}),
	\end{equation}
	while satisfying the system dynamics and constraints. The stage cost $ \ell(\cdot) $ is given as
	\begin{equation}
	\ell(x_{t},u_{t}) = x_{t}^\top Q x_{t} + u_{t}^\top R  u_{t},
	\end{equation}
	with $  Q \in \mb R^{ n \times n} $ and $  R \in \mb R^{m \times m} $ known positive semi-definite and positive definite matrices, respectively.
\end{subequations}

\begin{subequations}
	In this paper, we restrict our attention to constrained linear systems of the form \eqref{eq::system} whose matrices $ A $, $ B $, $ G $, $ H $, $ Q $ and $ R $ admit a structure which allow us to decompose the original system into an ordered set $ \mc M = \{1, \ldots, M\} $ of $ M $ dynamically coupled subsystems. In this context, the system states $ x_t $ and inputs $ u_t $ are decomposed as $ x_t = [x_{1,t}^\top,\ldots,x_{M,t}^\top]^\top $ and $ u_{t} = [u_{1,t}^\top, \ldots, u_{M,t}^\top]^\top $ where $ x_{i,t} \in \mb R^{n_i} $ and $ u_{i,t} \in \mb R^{m_i} $ denote the local states and inputs of $ i $-th subsystem, respectively. For each $ i $-th subsystem, we define the set $ \mc N_i \subseteq \mc M $ to include these of the subsystems whose states, $ x_{\Ni,t} \in \mb R^{n_\Ni} $, affect its dynamics and constraints. In addition, we define matrices $ U_i \in \{0,1\}^{n_i \times n} $, $ W_\Ni \in \{0,1\}^{n_\Ni \times n} $ and $ V_i \in \{0,1\}^{m_i \times m} $ such that
	\begin{equation}
	x_{i,t} = U_i x_t,\;x_{\Ni,t} = W_\Ni x_t \text{ and } u_{i,t} = V_i u_{t}.
	\end{equation}
	In this setting, the $ i $-th subsystem is defined by state dynamics 
	\begin{equation}
	x_{i,t+1} = A_\Ni x_{\Ni,t} + B_i u_{i,t},
	\end{equation} 
	constraints
	\begin{align}
	& x_{\Ni,t} \in \mc X_\Ni = \{ x_\Ni \in \mb R^{n_\Ni} \; : \; G_\Ni x_\Ni \le g_\Ni \},\\
	& u_{i,t} \in \mc U_i = \{ u_i \in \mb R^{m_i} \; : \; H_i u_i \le h_i \},
	\end{align}
	and objective function
	\begin{equation}
	J_{\infty} = \sum_{t=0}^{\infty} \sum_{i = 1}^M  \ell_i(x_{\Ni,t},u_{i,t}),
	\end{equation}
	where
	\begin{equation}\label{eq::StC}
	\ell_i(x_{\Ni,t},u_{i,t}) = x_{\Ni,t}^\top Q_{\Ni} x_{\Ni,t} + u_{i,t}^\top R_i  u_{i,t}.
	\end{equation} 
	The matrices $ A_\Ni $, $ B_i $, $ G_\Ni $, $ g_\Ni $, $ H_i $, $ h_i $, $ Q_\Ni $ and $ R_i $ are constructed from the problem data in \eqref{eq::system}, e.g., $ A_\Ni = U_i A W_\Ni^\top $ and $ B_i = U_i B V_i^\top $. For the splitting presented here, it is implicitly assumed that the subsystems are only coupled by their states and are decoupled in their inputs. This assumption is not restrictive and only introduced to simplify the exposition; it can easily be lifted by enriching the state space of each subsystem using auxiliary variables. 
\end{subequations}

\subsection{MPC formulation}

In the spirit of MPC, we introduce a value function $ V(\cdot) $ to upper bound the original infinite-horizon objective function by a finite-horizon one given as
\begin{equation*}
\widetilde{J}_{\infty} = V(x_{\Ti}) + \sum_{t \in \mc T} \ell(x_{t},u_{t}) \ge J_\infty,
\end{equation*}
where $ \mc T = \{0,\ldots,T-1\} $ and $ T $ denotes the prediction horizon. The value function, defined as
\begin{equation}\label{eq::vlFin}
V(x_{\Ti}) = x_{\Ti}^\top P_c x_{\Ti},
\end{equation}
where $ P_c $ is a positive definite matrix, upper approximates the cost of operating the system for all $ t \ge T $ under the terminal state feedback control law $ u_t = K_c x_t $. To satisfy state and input constraints for this terminal closed-loop system, we enforce $ x_\Ti $ to lie in a positively invariant set $ \mc X_{\Fi} $ being a subset of $ \mc X $.
\begin{definition} \label{def::PI}
	If for all $ x \in \mc X_\Fi \subseteq \mc X $ it holds
	\begin{equation*}
	(A+B K_c)x \in \mc X_{\Fi} \text{ and } K_c x \in \mc U
	\end{equation*}
	then the set $\mc X_\Fi$ is \textit{positively invariant} for the closed-loop system \eqref{eq::system} under the terminal controller $ u_t = K_c x_t $ for all $ t \ge T $.
\end{definition}

\begin{theorem}(\cite[\S 3]{Mayne2000})\label{thm::ST}
	If there exist functions $ \sigma^1(\cdot) $, $ \sigma^2(\cdot) $ and $ \sigma^3(\cdot) \in \mc K_\infty $ such that $ \forall x \in \mc X_\Fi $: 
	\begin{subequations}
		\begin{eqnarray}
		\sigma^1(\|x\|) \le V(x) \le \sigma^2(\|x\|) \label{eq::LPV1}\\
		\sigma^3(\|x\|) \le \ell(x,K_c x) \label{eq::LPV2}\\
		V((A+B K_c)x) - V(x) \le - \ell(x,K_c x) \label{eq::LPV3}
		\end{eqnarray}
	\end{subequations}
	then the function $ V(\cdot) $ is a Lyapunov function for the closed-loop system \eqref{eq::system} under the controller $ u_{t} = K_c x_{t} $ for all $ t \ge T $.
\end{theorem}
Notice that conditions \eqref{eq::LPV1} and \eqref{eq::LPV2} in Theorem \eqref{thm::ST} are satisfied by construction of value and stage cost functions. Condition \eqref{eq::LPV3} guarantees that the terminal controller stabilizes the system and $ V(\cdot) $ is an upper approximation of the true value function. The conditions in Theorem \ref{thm::ST} can be reformulated as LMIs by applying Schur complement techniques discussed in \cite{Boyd1994}. The resulting LMIs can efficiently be solved using numerical tools for semi-definite programming to compute matrices $ K_c $ and $ P_c $.

Having $ K_c $ and $ P_c $ computed, a positively invariant set $ \mc X_\Fi $ needs to be derived to guarantee state and input constraint satisfaction for the terminal closed-loop system. Ellipsoidal sets of the form
\begin{equation*}
\mc X_{\Fi} = \{ x \in \mb R^n : x^\top P_c x \le \alpha\},
\end{equation*}
where $ \alpha $ is a positive scalar, are common choices for positively invariant sets. This mainly stems from the fact that finding the maximum $ \alpha $ such that $ \mc X_\Fi $ respects the conditions in Definition \ref{def::PI} can be cast as a linear optimization problem \cite[\S 5.2]{Boyd1994}. However, these ellipsoidal sets can only provide an inner approximation of the maximum invariant set, $ \mc X_{\infty} $. For a linear and stable system with bounded constraint sets containing the origin, as the closed-loop system in \eqref{eq::system}, the maximum invariant set is given by a polyhedral set of the form
\begin{equation*}
\mc X_{\infty} = \{ x \in \mb R^n : A_\Fi x \le b_\Fi\}.
\end{equation*}
Here, the matrices $ A_\Fi $ and $ b_\Fi $ are calculated through an iterative procedure. Despite finite-time termination guarantees \cite{Gilbert1991}, it requires considerably higher computational effort than that of the ellipsoids. This makes the calculation of $ \mc X_{\infty} $ a hard task that is typically avoided for systems of large dimensions. In this centralized framework, the MPC optimization problem is given as follows:
\begin{equation}\label{Centralized}
\begin{array}{l}
\min~ V(x_{\Ti}) + \displaystyle \sum_{i \in \mc M} \left(\sum_{t \in \mc T}  \ell_i(x_{\Ni,t},u_{i,t})\right) \\[2ex]
\left.\begin{array}{r@{}l}\st~ & x_{i,t+1} = A_{\Ni} x_{\Ni,t} + B_{i}  u_{i,t}\\
& (x_{\Ni,t}, u_{i,t}) \in \mc X_\Ni \times \mc U_i\end{array}\right \rbrace \forall i\in\mathcal{M}\\
\phantom{\min~}x_{\Ti} \in {\mc X}_{\Fi}
\end{array} \tag{$\mc{C}$}
\end{equation}
with optimization variables $ (x_{t},u_{t}) $ for all $ t \in \mc T $. Problem \eqref{Centralized} is non-amendable to distributed computation algorithms since its value function $ V(\cdot) $ and terminal set $ \mc X_\Fi $ can admit a dense structure. This is because the computed $ K_c $ and $ P_c $ from solving \eqref{eq::LPV3} do not necessarily admit any distributed structure, even if the problem dynamics and constraints do.

To retain the distributed structure of the problem, the value function and terminal set of the MPC formulation need to respect the coupling structure of the system. To achieve this, the terminal controller of the $ i $-th subsystem is designed as
\begin{equation*}
u_{i,t} = K_\Ni x_{\Ni,t},\;\; \forall t \ge T,
\end{equation*} 
where $ K_\Ni \in \mb R^{m_i \times n_{\Ni}}$. The value function, now denoted by $ \wh{V}_\Fi(\cdot) $ to distinguish it from the non separable one in \eqref{eq::vlFin}, is given as
\begin{equation*}
\wh{V}(x_\Ti) = \sum_{i = 1}^{M} \wh{V}_i(x_{i,\Ti}),
\end{equation*}
with each $ \wh{V}_i(\cdot) $ being formulated as
\begin{equation*}
\wh{V}_i(x_{i,\Ti}) = x_{i,\Ti}^\top P_{i} x_{i,\Ti},
\end{equation*}
where $ P_{i} \in \mb R^{n_i \times n_i} $ is a positive definite matrix. Using similar techniques to \cite{Boyd1994}, the problem of finding $ K_\Ni $ and $ P_i $ for all $ i \in \mc M $ which fulfill the stability conditions of Theorem \ref{thm::ST} can be cast as a convex optimization problem involving LMIs. 

Similarly, the terminal set, now denoted by $ \wh{\mc X}_\Fi $, also need to admit a decoupled structure given as
\begin{equation*}
\wh{\mc X}_\Fi = \wh{\mc X}_{\Fi,1} \times \dots \times \wh{\mc X}_{\Fi,M}
\end{equation*}
with $ \wh{\mc X}_{\Fi,i} = \{x_{i} \in \mb R^{n_{i}}: x_{i}^\top P_{i} x_{i} \le \alpha_i \} $ (e.g., \cite{Conte2012,DarivianakisACC}) or $ \wh{\mc X}_{\Fi,i} = \{x_{i} \in \mb R^{n_{i}}: A_{\Fi,i} x_{i} \le b_{\Fi,i} \} $ (e.g., \cite{Keviczky2004,Riverso2014}).

In this distributed framework, the MPC optimization problem is given as follows:
\begin{equation}\label{Decentralized}
\begin{array}{l}
\min~ \displaystyle \sum_{i \in \mc M} \left(\wh{V}_i(x_{i,\Ti}) + \sum_{t \in \mc T}  \ell_i(x_{\Ni,t},u_{i,t})\right) \\[2ex]
\left.\begin{array}{r@{}l}\st~ & x_{i,t+1} = A_{\Ni} x_{\Ni,t} + B_{i}  u_{i,t},\\
& (x_{\Ni,t}, u_{i,t}) \in \mc X_\Ni \times \mc U_i,\\ 
& x_{i,\Ti} \in \wh{\mc X}_{\Fi,i}.\end{array}\right \rbrace \forall i\in\mathcal{M}
\end{array} \tag{$\mc{D}$}
\end{equation}
with optimization variables $ (x_{\Ni,t},u_{i,t}) $ for all $ i \in \mc M $, $ t \in \mc T $. Problem \eqref{Decentralized} exhibits the desired distributed structure which is amendable to distributed computation algorithms (e.g., the alternating method of multipliers \cite{Boyd2011}) to efficiently solve it. We emphasize that the quality of the generated solution greatly depends on the shape and size of the decoupled terminal sets. This is to say that if the considered terminal regions are small then the effort the system needs to take to push $ x_{\Ti} $ into these terminal sets can be large or in some instances even not feasible. The main reason for this conservativeness is attributed to the non-consideration of system constraints and current state in the design of $ K_\Ni $ and $ P_i $ for all $ i \in \mc M $. In current state-of-the art approaches, as in \cite{Mayne2005,Richards2007,Trodden2010,Farina2012,Riverso2014,Lucia2015,Trodden2017,Keviczky2004,Venkat2005,Stewart2010,Maestre2011,Giselsson2014,Conte2012,DarivianakisACC}, the design of stabilizing terminal controllers and value functions typically relies merely on satisfying the stability conditions of Theorem \ref{thm::ST}, while the computation of the terminal positively invariant sets is performed afterwards. In what follows, we propose a new machinery for the design of the terminal controllers based on the constraints and current state of the system. This allow us to couple the design of the stabilizing terminal controllers and the invariant terminal set computation under the same optimization problem.

\section{Adaptive distributed MPC}\label{sec::design}

\subsection{Invariant terminal sets}
For each $ i \in \mc M $, we consider ellipsoidal terminal sets of the form
\begin{equation*}
\wh{\mc X}_{\Fi,i}(\alpha_i) = \{x_{i} \in \mb R^{n_{i}}: x_{i}^\top Z_{i} x_{i} \le \alpha_i \},
\end{equation*}
where $ Z_i $ is a predefined positive definite matrix and $ \alpha_i $ is a scalar decision variable. To ease exposition, we define the decision variables matrices $ \alpha = \text{diag}(\alpha_1, \ldots, \alpha_M) $ and $ \alpha_\Ni = W_\Ni \alpha W_\Ni^\top $, and the invariant terminal set $ \wh{\mc X}_{\Fi,\Ni}(\alpha_\Ni) = \bigtimes_{j \in \Ni} \wh{\mc X}_{\Fi,j}(\alpha_i) $. The following proposition provides the necessary conditions for $ \wh{\mc X}_{\Fi,i}(\alpha_i) $ to be positively invariant.
\begin{proposition}\label{prop::PI}
	If for each subsystem $ i \in \mc M $ it holds that $ \forall x_{\Ni} \in \wh{\mc X}_{\Fi,\Ni}(\alpha_\Ni) $ then
	\begin{subequations}
		\begin{eqnarray}
		(A_\Ni +B_i K_\Ni)x_{\Ni} \in \wh{\mc X}_{\Fi,i}(\alpha_i), \label{prop::PI::eq1}\\
		x_\Ni \in \mc X_\Ni,\label{prop::PI::eq2}\\
		K_\Ni x_{\Ni} \in \mc U_i,\label{prop::PI::eq3}
		\end{eqnarray}
	\end{subequations}
	then each set $ \wh{\mc X}_{\Fi,i}(\alpha_i) $ is positively invariant; hence, $ \wh{\mc X}_\Fi(\alpha) = \wh{\mc X}_{\Fi,1}(\alpha_1) \times \cdots \times \wh{\mc X}_{\Fi,M}(\alpha_M) $ is also positively invariant.
\end{proposition}
Roughly speaking, the conditions of the above proposition are equivalent to assuming that each $ i $-th subsystem treats the states of its neighboring subsystems as disturbances to its own. Under this assumption the terminal set $ \wh{\mc X}_{\Fi,i}(\alpha_i) $ can be considered as robust positively invariant set and the terminal controller $ K_\Ni $ as a disturbance feedback one. In the sequel, we provide the reformulations of the robust constraints in Proposition \ref{prop::PI}. 
\begin{theorem}
	For each $ i \in \mc M $, condition 
	\begin{equation*} 
	(A_\Ni + B_i K_\Ni) x_\Ni \in \wh{\mc X}_{\Fi,i}(\alpha_i) \text{ for all } x_\Ni \in \wh{\mc X}_{\Fi,\Ni}(\alpha_\Ni),
	\end{equation*}
	holds if $ \exists \lambda_{ij} \ge 0 $ with $ j \in \mc N_i $ such that
	\begin{equation} \label{eq::C1}
	\left[\begin{array}{@{}cc@{}}
	Z_i^{-1}\alpha_i^{1/2} & (A_\Ni \alpha^{1/2}_\Ni + B_i K_\Ni \alpha^{1/2}_\Ni) \\
	(A_\Ni \alpha^{1/2}_\Ni + B_i K_\Ni \alpha^{1/2}_\Ni)^\top & \sum_{j\in \Ni} \lambda_{ij} Z_{ij}
	\end{array}\right] \succeq 0 \text{ and } \sum_{j\in \Ni} \lambda_{ij} \le \alpha_i^{1/2},
	\end{equation}
	where $ Z_{ij} = W_\Ni U_j^\top Z_j U_j W_\Ni^\top $.
\end{theorem}
\begin{proof}
	Parsing the expression of invariance for the $ i $-th subsystem, we get
	\begin{equation*}
	\begin{array}{l}
	x_{\Ni}^\top (A_\Ni + B_i K_\Ni)^\top Z_i (A_\Ni + B_i K_\Ni) x_{\Ni} \le \alpha_i \text{ for all } x_{j}^\top Z_j x_{j} \le \alpha_j \text{ with } j \in \mc N_i.
	\end{array}
	\end{equation*}
	We use the auxiliary variable $ s_i \in \mb R^{n_i} $ to make the substitution $ x_i = \alpha_i^{1/2} s_i $. Using this, the robust constraint above is equivalently written as
	\begin{equation*}
	\begin{array}{l}
	s_{\Ni}^\top (A_\Ni\alpha_\Ni^{1/2} + B_i K_\Ni\alpha_\Ni^{1/2})^\top Z_i (A_\Ni\alpha_\Ni^{1/2} + B_i K_\Ni\alpha_\Ni^{1/2}) s_{\Ni} \le \alpha_i \text{ for all } s_{j}^\top Z_j s_{j} \le 1 \text{ with } j \in \mc N_i \Leftrightarrow \\
	s_{\Ni}^\top (A_\Ni\alpha_\Ni^{1/2} + B_i K_\Ni\alpha_\Ni^{1/2})^\top Z_i \alpha_i^{-1/2} (A_\Ni\alpha_\Ni^{1/2} + B_i K_\Ni\alpha_\Ni^{1/2}) s_{\Ni} \le \alpha_i^{1/2} \text{ for all } s_{\Ni}^\top Z_{ij} s_{\Ni} \le 1 \text{ with } j \in \mc N_i.
	\end{array}
	\end{equation*}
	Now using the S-lemma \cite[\S 2.6.3]{Boyd1994}, the robust constraint above  holds if $ \exists \lambda_{ij} \ge 0 $ with $ j \in \mc N_i $ such that
	\begin{equation*}
	(A_\Ni\alpha_\Ni^{1/2} + B_i K_\Ni\alpha_\Ni^{1/2})^\top Z_i \alpha_i^{-1/2} (A_\Ni\alpha_\Ni^{1/2} + B_i K_\Ni\alpha_\Ni^{1/2}) \preceq \sum_{j\in \Ni} \lambda_{ij} Z_{ij},
	\end{equation*}
	and 
	\begin{equation*}
	\sum_{j\in \Ni} \lambda_{ij} \le \alpha_i^{1/2}.
	\end{equation*}
	By applying the Schur-complement the proof is concluded.
\end{proof}

We continue by providing tractable approximations to conditions \eqref{prop::PI::eq2} and \eqref{prop::PI::eq3} of Proposition \ref{prop::PI} which guarantee that the state and input constraints of the system are satisfied by the terminal controller. To do so, we denote the $ \ell $-th row (out of $  n_{g,i} $ rows) of the $ G_\Ni $ and $ g_\Ni $ state constraint matrices by $ G_{\Ni}^\ell $ and $ g_{\Ni}^\ell $, respectively. Similarly, we denote the $ \ell $-th row (out of $  n_{h,i} $ rows) of the $ H_\Ni $ and $ h_i $ input constraint matrices by $ H_{\Ni}^\ell $ and $ h_{i}^\ell $, respectively. 
\begin{theorem}
	For each $ i \in \mc M $, the $ \ell $-th state constraint
	\begin{equation*}
	G_{\Ni}^\ell x_{\Ni} \le g_{\Ni}^\ell \text{ for all } x_\Ni \in \wh{\mc X}_{\Fi,\Ni}(\alpha_\Ni),
	\end{equation*}
	with $ \ell = 1, \ldots, n_{g,i} $ holds if $ \exists \tau_{ij}^\ell \ge 0 $ with $ j \in \mc N_i $ such that
	\begin{equation} \label{eq::C2}
	\begin{bmatrix}
	g^\ell_{\Ni} & G^{\ell}_{\Ni}\alpha^{1/2}_\Ni \\
	\alpha^{1/2}_\Ni G^{\ell\,\top}_{\Ni} & \sum_{j\in \Ni} \tau_{ij}^\ell Z_{ij}
	\end{bmatrix} \succeq 0 \text{ and} \sum_{j\in \Ni} \tau_{ij}^\ell \le g^\ell_{\Ni}.
	\end{equation}
	Similarly, the $ \ell $-th input constraint
	\begin{equation*}
	H_i^\ell K_\Ni x_{\Ni} \le h^\ell_i \text{ for all } x_\Ni \in \wh{\mc X}_{\Fi,\Ni}(\alpha_\Ni),
	\end{equation*}
	with $ \ell = 1, \ldots, n_{h,i} $ holds if $ \exists \rho_{ij}^\ell \ge 0 $ with $ j \in \mc N_i $ such that
	\begin{equation}\label{eq::C3}
	\begin{bmatrix}
	h^\ell_{i} & H^{\ell}_{i} K_\Ni \alpha^{1/2}_\Ni \\
	\alpha^{1/2}_\Ni K_\Ni^\top H^{\ell\,\top}_{i} &  \sum_{j\in \Ni} \rho_{ij}^\ell Z_{ij}
	\end{bmatrix} \succeq 0 \text{ and} \sum_{j\in \Ni} \rho_{ij}^\ell \le h^\ell_{i},
	\end{equation}
	where $ Z_{ij} = W_\Ni U_j^\top Z_j U_j W_\Ni^\top $.
\end{theorem}
\begin{proof}
	Parsing the expression \eqref{prop::PI::eq2} for the state constraints we get
	\begin{equation*}
	G^\ell_{\Ni} x_{\Ni} \le g^\ell_{\Ni} \text{ for all } x_{j}^\top Z_{j} x_{j} \le \alpha_j \text{ with } j \in \mc N_i.
	\end{equation*}
	We introduce the auxiliary variable $ s_i \in \mb R^{n_i} $ to make the substitution $ x_i = \alpha_i^{1/2} s_i $. Using this, the robust constraint above  is equivalently written as
	\begin{equation*}
	G^\ell_{\Ni} \alpha^{1/2}_\Ni s_{\Ni,t} \le g^\ell_{\Ni} \text{ for all } s_{j}^\top Z_{j} s_{j} \le 1 \text{ with } j \in \mc N_i.
	\end{equation*}
	It is easy to verify that in case of ellipsoidal sets the robust constraint above  is equivalent to 
	\begin{equation*}
	\begin{array}{l}
	\|G^\ell_{\Ni} \alpha^{1/2}_\Ni s_{\Ni,t}\|_2 \le g^\ell_{\Ni} \text{ for all } s_{j}^\top Z_{j} s_{j} \le 1 \text{ with } j \in \mc N_i \Leftrightarrow \\
	s_{\Ni,t}^\top \alpha^{1/2}_\Ni G^{\ell\,\top}_{\Ni} (g^{\ell})^{-1}_{i} G^j_{\Ni} \alpha^{1/2}_\Ni s_{\Ni,t}\le g^\ell_{\Ni} \text{ for all } s_{\Ni}^\top Z_{ij} s_{\Ni} \le 1 \text{ with } j \in \mc N_i. \\
	\end{array}
	\end{equation*}
	Applying the S-lemma, this robust constraint holds if $ \exists \tau_{ij}^\ell \ge 0 $ with $ j \in \mc N_i $ such that
	\begin{equation*}
	\begin{bmatrix}
	g^\ell_{\Ni} & 0 \\
	0 & -\alpha^{1/2}_\Ni G^{\ell\,\top}_{\Ni} g^{\ell,-1}_{i} G^\ell_{\Ni} \alpha^{1/2}_\Ni
	\end{bmatrix} \succeq \sum_{j \in \mc N_i} \tau_{ij}^\ell \begin{bmatrix}
	1 & 0 \\
	0 & - Z_{ij}
	\end{bmatrix}
	\end{equation*} 
	Then, we apply the Schur complement to obtain \eqref{eq::C2}.
	Following the exact similar derivation arguments, one can prove that the equivalent of the $ \ell $-th input constraint in \eqref{prop::PI::eq3}, given as
	\begin{equation*} 
	H_i^\ell K_\Ni \alpha^{1/2}_\Ni s_{\Ni,t} \le h^\ell_i \text{ for all } s_{j}^\top Z_{j} s_{j} \le 1 \text{ with } j \in \mc N_i, \\
	\end{equation*}
	with $ \ell = 1, \ldots, n_{h,i} $, holds if \eqref{eq::C3} is satisfied. This concludes the proof.
\end{proof}

\subsection{Stability of terminal closed-loop system}

To ensure stability of the terminal closed-loop system \eqref{eq::system} under the control law $ u_{i,t} = K_\Ni x_{\Ni,t} $ for all $ i \in \mc M $ and $ t \ge T $, the conditions of Theorem \ref{thm::ST} need to be satisfied. The non decoupled structure of these conditions makes them unsuitable for explicit consideration in the formulation of Problem \eqref{Decentralized}. Instead, we adopt the notion of structured control Lyapunov functions, introduced in \cite{Jokic2009}, which allow us to consider the conditions for stability in a way that respects the distributed structure of our system.
\begin{theorem}(\cite[\S 3.2]{Jokic2009})\label{thm::StD} 
	If for each $ i \in \mc M $ there exist functions $ \sigma_i^1(\cdot) $, $ \sigma_i^2(\cdot) $ and $ \sigma_i^3(\cdot) \in \mc K_\infty $ such that $ \forall x_i \in \wh{\mc X}_{\Fi,i}(\alpha_i) $:
	\begin{subequations}
		\begin{eqnarray}
		\sigma_i^1(\|x_i\|) \le \wh{V}_{i}(x_i) \le \sigma_i^2(\|x_i\|) \label{eq::LPVD1}\\
		\sigma_i^3(\|x_i\|) \le \ell_i(x_\Ni,K_\Ni x_\Ni) \label{eq::LPVD2} \\
		\wh{V}_{i}((A_\Ni + B_i K_\Ni)x_\Ni) - \wh{V}_{i}(x_i) \le - \ell_i(x_\Ni,K_\Ni x_\Ni) + \gamma_i(x_\Ni) \label{eq::LPVD3} \\
		\sum_{i=1}^M \gamma_i(x_\Ni) \le 0  \label{eq::LPVD4}
		\end{eqnarray}
	\end{subequations}
	then the function $ \wh{V}(x_{i,t}) = \sum_{i=1}^M \wh{V}_{i}(x_{i,t}) $ is a Lyapunov function for the closed-loop system \eqref{eq::system} under the terminal controllers $ u_{i,t} = K_\Ni x_{\Ni,t} $ for all $ i \in \mc M $ and $ t \ge T $.
\end{theorem}
Notice that Theorem \ref{thm::StD} implies the more general Lyapunov stability Theorem \ref{thm::ST}. Conditions \eqref{eq::LPVD1} and \eqref{eq::LPVD2} in Theorem \eqref{thm::StD} are satisfied by construction of value functions and stage costs. Condition \eqref{eq::LPVD3} together with condition \eqref{eq::LPVD4} guarantees that the terminal controllers stabilize the system and $ \wh{V}(\cdot) $ is an upper approximation of the true value function. Note that the definition above does not impose that each function $ \wh{V}_{i}(\cdot) $ is a control Lyapunov function for the corresponding subsystem in $ \wh{\mc X}_{\Fi,i}(\alpha_i) $. Roughly speaking, this condition allows a local terminal cost to increase as at the same time the sum of all terminal cost in \eqref{eq::LPVD4} decreases. Consider for instance two interconnected subsystems where state $ x_{1,t} $ for subsystem~$ 1 $ at time $ t $ rests at the origin, i.e., $ \wh{V}_{1}(x_{1,t}=0) = 0 $. If the state $ x_{2,t} $ of subsystem $ 2 $ is nonzero, then $ x_{1,t+1} $ will necessarily be driven away from the origin as soon as the controller $ K_{{}_{\mc N_1}} $ can not fully dissipate the effect of $ x_{2,t} $ on subsystem~$ 1 $; causing $ \wh{V}_{1}(x_{1,t+1}\neq 0) $ to increase but $\wh{V}_{1}(x_{1,t+1})$ has to increase less than $ \wh{V}_{2}(x_{2,t+1}) $ decreases, i.e., $ \wh{V}_{2}(x_{1,t+1})- \wh{V}_{2}(x_{1,t}) < \wh{V}_{1}(x_{1,t+1})- \wh{V}_{1}(x_{1,t}) $, such that $ \wh{V}(\cdot) $ is an overall Lyapunov function.

\begin{theorem}
	Conditions \eqref{eq::LPVD3} and \eqref{eq::LPVD4} hold if $ \exists H_\Ni \in \mb R^{n_\Ni \times n_\Ni},\, Y_\Ni \in \mb R^{m_i \times n_\Ni} $ such that
	\begin{equation}\label{eq::C4}
	\begin{bmatrix}
	P_i^{-1} \alpha^{1/2}_i & A_\Ni\alpha^{1/2}_\Ni + B_i Y_\Ni & 0 & 0 \\
	(A_\Ni\alpha^{1/2}_\Ni + B_i Y_\Ni)^\top & P_{ii} \alpha^{1/2}_i + H_\Ni &  \alpha^{1/2}_\Ni Q_{\Ni}^{1/2}  & Y_\Ni^\top R_i^{1/2} \\
	0 & Q_{\Ni}^{1/2} \alpha^{1/2}_\Ni  & \alpha^{1/2}_i I_\Ni & 0 \\
	0 & R_i^{1/2} Y_\Ni  & 0 & \alpha^{1/2}_i
	\end{bmatrix} \succeq 0
	\end{equation}
	and
	\begin{equation}\label{eq::C5}
	\sum_{i = 1}^M W_{\Ni}^\top H_\Ni W_{\Ni} \preceq 0
	\end{equation}
	where $ P_{ii} = W_\Ni U_i^\top P_i U_i W_\Ni^\top $.
\end{theorem}
\begin{proof}
	Conditions \eqref{eq::LPVD3} is written as: For all $ i \in \mc M $, 
	\begin{equation*}
	V_i((A_\Ni + B_i K_\Ni)x_\Ni) - V_i(x_{i}) \le - \ell_i(x_{\Ni},K_\Ni x_\Ni) + \gamma_i(x_{\Ni}) \text{ for all } x_{j}^\top Z_{j} x_{j} \le \alpha_j \text{ with } j \in \mc N_i.
	\end{equation*}
	This robust inequality is expanded as,
	\begin{equation*}
	x_{\Ni}^\top \big(P_{ii} -(A_\Ni + B_i K_\Ni)^\top P_i (A_\Ni + B_i K_\Ni) -Q_{\Ni} -K_\Ni^\top R_i K_\Ni + \Gamma_{\Ni} \big) x_{\Ni} \ge 0 \text{ for all } x_{j}^\top Z_{j} x_{j} \le \alpha_j \text{ with } j \in \mc N_i.
	\end{equation*}
	We now use the auxiliary variable $ s_i \in \mb R^{n_i} $ to make the substitution $ x_i = \alpha_i^{1/2} s_i $. Using this, the robust constraint above is equivalently written as 
	\begin{equation*}
	\begin{array}{l}
	s_{\Ni}^\top \big(\alpha^{1/2}_\Ni P_{ii}\alpha^{1/2}_\Ni -(A_\Ni\alpha^{1/2}_\Ni + B_i K_\Ni\alpha^{1/2}_\Ni)^\top P_i (A_\Ni\alpha^{1/2}_\Ni + B_i K_\Ni\alpha^{1/2}_\Ni) \\ \hspace*{1cm}-\alpha^{1/2}_\Ni Q_{\Ni} \alpha^{1/2}_\Ni -(K_\Ni \alpha^{1/2}_\Ni)^\top R_i (K_\Ni \alpha^{1/2}_\Ni) + \alpha^{1/2}_\Ni \Gamma_{\Ni} \alpha^{1/2}_\Ni \big) s_{\Ni} \ge 0 \text{ for all } s_{\Ni}^\top Z_{ij} s_{\Ni} \le 1 \text{ with } j \in \mc N_i.
	\end{array}
	\end{equation*}
	Making use of $ \alpha^{1/2}_\Ni {P}_{ii} \alpha^{1/2}_\Ni = \alpha_i^{1/2} {P}_{ii} \alpha_i^{1/2} $, we have that
	\begin{equation*}
	\begin{array}{l}
	s_{\Ni}^\top \big(\alpha^{1/2}_i {P}_{ii}\alpha^{1/2}_i -(A_\Ni\alpha^{1/2}_\Ni + B_i K_\Ni\alpha^{1/2}_\Ni)^\top P_i (A_\Ni\alpha^{1/2}_\Ni + B_i K_\Ni\alpha^{1/2}_\Ni) \\ \hspace*{1cm}-\alpha^{1/2}_\Ni Q_{\Ni} \alpha^{1/2}_\Ni -(K_\Ni \alpha^{1/2}_\Ni)^\top R_i (K_\Ni \alpha^{1/2}_\Ni) + \alpha^{1/2}_\Ni \Gamma_{\Ni} \alpha^{1/2}_\Ni \big) s_{\Ni} \ge 0 \text{ for all } s_{\Ni}^\top Z_{ij} s_{\Ni} \le 1 \text{ with } j \in \mc N_i.
	\end{array}
	\end{equation*}
	Applying the S-lemma the robust constraint above holds if $ \exists \phi_{ij} \ge 0 $ with $ j \in \mc N_i $ such that:
	\begin{equation*}
	\left[\begin{array}{@{}c|c@{}}
	0 & 0 \\ \hline
	0 & \begin{array}{l}
	P_{ii}\alpha^{1/2}_i -(A_\Ni\alpha^{1/2}_\Ni + B_i K_\Ni\alpha^{1/2}_\Ni)^\top P_i \alpha^{-1/2}_i (A_\Ni\alpha^{1/2}_\Ni + B_i K_\Ni\alpha^{1/2}_\Ni) \\ \hspace*{.1cm}-\alpha^{1/2}_\Ni Q_{\Ni}\alpha^{-1/2}_i \alpha^{1/2}_\Ni -(K_\Ni \alpha^{1/2}_\Ni)^\top R_i\alpha^{-1/2}_i (K_\Ni \alpha^{1/2}_\Ni) + \alpha^{1/2}_\Ni \Gamma_{\Ni} \alpha^{-1/2}_i \alpha^{1/2}_\Ni
	\end{array}
	\end{array}\right] \succeq \sum_{j \in \mc N_i} \phi_{ij} \begin{bmatrix}
	1 & 0 \\
	0 & - Z_{ij}
	\end{bmatrix}
	\end{equation*} 
	This implies $ \phi_{ij} = 0 $ for all $ j \in \mc N_i $; hence, the matrix inequality constraint above is equivalent written
	\begin{equation*}
	\begin{array}{l}
	P_{ii}\alpha^{1/2}_i -(A_\Ni\alpha^{1/2}_\Ni + B_i K_\Ni\alpha^{1/2}_\Ni)^\top P_i \alpha^{-1/2}_i (A_\Ni\alpha^{1/2}_\Ni + B_i K_\Ni\alpha^{1/2}_\Ni) \\ \hspace*{1cm}-\alpha^{1/2}_\Ni Q_{\Ni}\alpha^{-1/2}_i \alpha^{1/2}_\Ni -(K_\Ni \alpha^{1/2}_\Ni)^\top R_i\alpha^{-1/2}_i (K_\Ni \alpha^{1/2}_\Ni) + \alpha^{1/2}_\Ni \Gamma_{\Ni} \alpha^{-1/2}_i \alpha^{1/2}_\Ni \succeq 0
	\end{array}
	\end{equation*}
	We use the Schur complement to write this expression as,
	\begin{equation*}
	\begin{array}{l}
	\begin{bmatrix}
	P^{-1}_i \alpha^{1/2}_i & (A_\Ni\alpha^{1/2}_\Ni + B_i K_\Ni\alpha^{1/2}_\Ni) \\
	(A_\Ni\alpha^{1/2}_\Ni + B_i K_\Ni\alpha^{1/2}_\Ni)^\top & P_{ii}\alpha^{1/2}_i + \alpha^{1/2}_\Ni \Gamma_{\Ni} \alpha^{-1/2}_i \alpha^{1/2}_\Ni
	\end{bmatrix} \\ \hspace*{3cm} - \begin{bmatrix}
	0 & 0 \\
	\alpha^{1/2}_\Ni Q_{\Ni}^{1/2}  &  (K_\Ni \alpha^{1/2}_\Ni)^\top R^{1/2}
	\end{bmatrix} \begin{bmatrix}
	\alpha^{-1/2}_i I_\Ni & 0 \\
	0 & \alpha^{-1/2}_i 
	\end{bmatrix} \begin{bmatrix}
	0 & Q_{\Ni}^{1/2} \alpha^{1/2}_\Ni \\
	0 &  R^{1/2} K_\Ni \alpha^{1/2}_\Ni 
	\end{bmatrix} \succeq 0.
	\end{array}
	\end{equation*}
	Applying once again the Schur complement, lead to
	\begin{equation*}
	\begin{bmatrix}
	P_i^{-1} \alpha^{1/2}_i & A_\Ni\alpha^{1/2}_\Ni + B_i K_\Ni\alpha^{1/2}_\Ni & 0 & 0 \\
	(A_\Ni\alpha^{1/2}_\Ni + B_i K_\Ni\alpha^{1/2}_\Ni)^\top & P_{ii} \alpha^{1/2}_i + \alpha^{1/2}_\Ni \Gamma_{\Ni} \alpha^{-1/2}_i \alpha^{1/2}_\Ni &  \alpha^{1/2}_\Ni Q_{\Ni}^{1/2}  & (K_\Ni \alpha^{1/2}_\Ni)^\top R_i^{1/2} \\
	0 & Q_{\Ni}^{1/2} \alpha^{1/2}_\Ni  & \alpha^{1/2}_i I_\Ni & 0 \\
	0 & R_i^{1/2} K_\Ni \alpha^{1/2}_\Ni  & 0 & \alpha^{1/2}_i
	\end{bmatrix} \succeq 0.
	\end{equation*}
	To obtain an LMI expression, we make the substitutions $ Y_\Ni = K_\Ni \alpha^{1/2}_\Ni $ and $ H_\Ni = \alpha^{1/2}_\Ni \Gamma_{\Ni} \alpha^{-1/2}_i \alpha^{1/2}_\Ni $ so that the matrix inequality above can be written as an LMI of the form \eqref{eq::C4}.
	
	Finally, conditions \eqref{eq::LPVD4} is written as:
	\begin{equation*}
	\begin{array}{l}
	\sum_{i = 1}^M \gamma_i(x_{\Ni}) \le 0 \text{ for all } x_{i}^\top Z_{i} x_{i} \le \alpha_i \text{ with } i = 1,\ldots,M \Leftrightarrow\\
	x^\top \big(\sum_{i = 1}^M W_{\Ni}^\top \Gamma_{\Ni} W_{\Ni} \big) x \le 0 \text{ for all } x_{i}^\top Z_{i} x_{i} \le \alpha_i \text{ with } i = 1,\ldots,M \Leftrightarrow\\
	s^\top \big(\sum_{i = 1}^M \alpha W_{\Ni}^\top \Gamma_{\Ni} W_{\Ni} \alpha \big) s \le 0 \text{ for all } s_{i}^\top Z_{i} s_{i} \le 1 \text{ with } i = 1,\ldots,M \Leftrightarrow\\
	s^\top \big(\sum_{i = 1}^M  W_{\Ni}^\top \alpha_\Ni \Gamma_{\Ni} \alpha_\Ni W_{\Ni} \big) s \le 0 \text{ for all } s_{i}^\top Z_{i} s_{i} \le 1 \text{ with } i = 1,\ldots,M \Leftrightarrow\\
	s^\top \big(\sum_{i = 1}^M  W_{\Ni}^\top \alpha_\Ni \Gamma_{\Ni} \alpha^{-1/2}_i \alpha_\Ni W_{\Ni} \big) s \le 0 \text{ for all } s_{i}^\top Z_{i} s_{i} \le 1 \text{ with } i = 1,\ldots,M \Leftrightarrow\\
	s^\top \big(\sum_{i = 1}^M  W_{\Ni}^\top H_\Ni W_{\Ni} \big) s \le 0 \text{ for all } s_{i}^\top Z_{i} s_{i} \le 1 \text{ with } i = 1,\ldots,M
	\end{array}
	\end{equation*}
	Using the S-lemma, we have that the robust constraints above holds if $ \exists \phi_{ij} \ge 0 $ with $ j \in \mc M $ such that:
	\begin{equation*}
	\begin{bmatrix}
	0 & 0 \\
	0 & \sum_{i = 1}^M  W_{\Ni}^\top H_\Ni W_{\Ni}
	\end{bmatrix} \succeq \sum_{j \in \mc M} \phi_{ij} \begin{bmatrix}
	1 & 0 \\
	0 & - U_j^\top Z_j U_j
	\end{bmatrix} 
	\end{equation*} 
	This implies $ \phi_{ij} = 0 $ for all $ j \in \mc N_i $; hence, the matrix inequality constraint above can be written as an LMI of the form \eqref{eq::C5} which concludes the proof.
\end{proof}

To be compatible with the LMI invariant and stability conditions presented previously, we rewrite condition $ x_{i,\Ti} \in \wh{\mc X}_{\Fi,i}(\alpha_i) $ in terms of the square root of the decision variable $ \alpha^{1/2}_i $, as follows:
\begin{equation}\label{eq::C6}
x_{i,\Ti} \in \wh{\mc X}_{\Fi,i}(\alpha_i) \Leftrightarrow \begin{bmatrix}
\alpha^{1/2}_i Z_i & x_{i,\Ti} \\
x_{i,\Ti}^\top & \alpha^{1/2}_i
\end{bmatrix} \succeq 0,
\end{equation}
where the Schur complement is applied. To this end, we define for each $ i \in \mc M $ the set
\begin{equation*}
\wt{\mc X}_{\Fi,i}(\alpha_\Ni) = \{x_i \in \mb R^{n_i} \,:\, \text{ Conditions } \eqref{eq::C2}, \eqref{eq::C3}, \eqref{eq::C4}, \eqref{eq::C5},  \eqref{eq::C6} \text{ hold.} \}.
\end{equation*}

\subsection{Stability and recursive feasibility}

The decentralized MPC problem with adaptive terminal sets is given as
\begin{equation}\label{DecentralizedAdaptive}
\begin{array}{l}
\min~ \displaystyle \sum_{i \in \mc M} \left(\wh{V}_i(x_{i,\Ti}) + \sum_{t \in \mc T}  \ell_i(x_{\Ni,t},u_{i,t})\right) \\[2ex]
\left.\begin{array}{r@{}l}\st~ & x_{i,t+1} = A_{\Ni} x_{\Ni,t} + B_{i}  u_{i,t}\\
& (x_{\Ni,t}, u_{i,t}) \in \mc X_\Ni \times \mc U_i\\ 
& x_{i,\Ti} \in \wt{\mc X}_{\Fi,i}(\alpha_\Ni)\end{array}\right \rbrace \forall i\in\mathcal{M},
\end{array}\tag{$\mc{DA}$}
\end{equation}
with optimization variables $ (x_{\Ni,t}, u_{i,t}, \alpha_\Ni) $ for all $ i \in \mc M $, $ t \in \mc T $.

Following the reasoning in \cite{Mayne2000}, we now show that establishing stability and recursive feasibility for the closed-loop system \eqref{eq::system} under the MPC controller defined in Problem \eqref{DecentralizedAdaptive} is equivalent to requiring that this problem is initially feasible.
\begin{proposition}
	The MPC Problem \eqref{DecentralizedAdaptive} with adaptive terminal sets is recursively feasible. Moreover, the closed-loop system \eqref{eq::system} resulting from applying the MPC controller defined by Problem \eqref{DecentralizedAdaptive} in a receding horizon fashion is asymptotically stable.
\end{proposition}
\begin{proof}
	Assume that the optimization Problem \eqref{DecentralizedAdaptive} is feasible at time $ t = t_0 $. Then, we obtain a sequence of optimal inputs $ [ u_{i,t_0}, \ldots,  u_{i,t_0+T-1}] $ for all $ i \in \mc M $ which satisfy the state, input and terminal constraints of the problem. Since $ \wt{\mc X}_{\Fi,i}(\alpha_\Ni) $ is a positively invariant region, the sequence of inputs $ [u_{i,t_0+1}, \ldots, u_{i,t_0+T-1}, K_\Ni x_{\Ni,t_0+T}] $, often referred as ``tail'' sequence, can be verified to be a feasible solution for Problem \eqref{DecentralizedAdaptive} at time $ t = t_0 + 1 $. Hence, if the optimization Problem \eqref{DecentralizedAdaptive} has a solution at time $ t_0 $ then it is guaranteed to have a solution at time $ t_0+1 $ establishing this way recursive feasibility. Since any solution to the MPC Problem \eqref{DecentralizedAdaptive} enforces the terminal set to be control invariant, recursive feasibility is preserved even though the terminal sets are adapting on the current state of the system.
	
	To prove stability of Problem \eqref{DecentralizedAdaptive}, define the objective function cost $ J_{t_0} $ at time $ t_0 $ as 
	\begin{equation*}
	J_{t_0} = \sum_{i \in \mc M} \left(\wh{V}_i(x_{i,\Ti}) + \sum_{t = t_0}^{t_0+T-1} \ell_i(x_{\Ni,t},u_{i,t})\right).
	\end{equation*}
	Let now $ J^*_{t_0} $ be the cost at time $ t_0 $ when applying the optimal sequence and $ \wh{J}_{t_0+1} $ be the cost associated with applying the ``tail'' sequences from time $ t_0+1 $. Then, we have that
	\begin{equation*}
	\wh{J}_{t_0+1} - J^*_{t_0} = - \ell(x_{t_0}, u_{t_0}).
	\end{equation*}
	Noting that $ \wh{J}_{t_0+1} \ge {J}^*_{t_0+1} $ due to the suboptimality of the tail sequences gives
	\begin{equation*}
	J^*_{t_0+1} - J^*_{t_0} \le - \ell(x_{t_0}, u_{t_0})
	\end{equation*}
	implying that $ J^* $ is a Lyapunov function for the system; hence, the closed-loop system \eqref{eq::system} resulting from applying the MPC controller defined by Problem \eqref{DecentralizedAdaptive} in a receding horizon fashion is asymptotically stable.
\end{proof}

We close this section by remarking that we distinguish design and online phases for the implementation of the proposed DMPC scheme. During the design phase, the value functions $ \wh{V}_i(\cdot) $ for all $ i \in \mc M $ are calculated, while during the online phase the DMPC Problem \eqref{DecentralizedAdaptive} is solved. Both phases are amendable to distributed computation, e.g., see \cite{Conte2012} for a distributed algorithm to calculated $ \wh{V}_i(\cdot) $ and \cite{Boyd2011} for consensus ADMM as a suitable distributed algorithm to solve the DMPC Problem \eqref{DecentralizedAdaptive}.

\section{Numerical examples}\label{sec::numericals}

In this section, we conduct a number of simulation-based studies to assess the efficacy of the proposed DMPC formulation with adaptive terminal sets. We focus our attention on two examples: $ (i) $ an illustrative two-dimensional system that allow us to assess, numerically and graphically, the benefits of invariant terminal sets that can adapt on the current and predicted states of the system, and $ (ii) $ a series of masses that are connected by springs and dampers which are suitable for studying the scalability and the closed-loop behavior of the proposed methodology.

\subsection{Illustrative example} \label{sec:: P1}
We consider a linear time-invariant system with dynamics
\begin{subequations}\label{eq::ToyExample}
	\begin{equation}
	x_{t+1} = A x_t + B u_t,
	\end{equation}
	where $ x_t \in \mb R^2 $ denote the states and $ u_t \in \mb R^2 $ the inputs. The system matrices $ A \in \mb R^{2 \times 2} $ and $ B \in \mb R^{2\times 2} $ are given as
	\begin{equation}
	A = \begin{bmatrix}
	5 & 0.1 \\
	0.3 & 0.9
	\end{bmatrix} \text{ and } B = \text{diag}([1, 1]),
	\end{equation}
	respectively. The system is subject to linear state and input constraints
	\begin{equation}
	[-5, -5]^\top \le x_t \le [5, 5]^\top \text{ and } [-1, -1]^\top \le u_t \le [1, 1]^\top
	\end{equation}
	and its goal is to minimize the infinite-horizon objective function
	\begin{equation}
	J_\infty  = \sum_{t = 0}^\infty x_t^\top Q x_t + u_t^\top R u_t,
	\end{equation}
\end{subequations}
where $ Q = \text{diag}([1, 1]) $ and $ R = \text{diag}([0.1, 0.1]) $. We split the system into two dynamically coupled subsystems with states $ x_{1,t}, x_{2,t} \in \mb R $ and inputs $ u_{1,t}, u_{2,t} \in \mb R $ such that $ x_t = [x_{1,t}, x_{2,t}]^\top $ and $ u_t = [u_{1,t}, u_{2,t}]^\top $. The dynamics, constraints and objective functions of these subsystems can straightforwardly be constructed through \eqref{eq::ToyExample}. 

We approximate the infinite-horizon objective function as
\begin{equation*}
\widetilde{J}_\infty  = V(x_{\Ti}) + \sum_{t = 0}^{T-1} x_t^\top Q x_t + u_t^\top R u_t.
\end{equation*} 
where $ V(\cdot) $ denotes the value function. In the centralized MPC formulation $ V(\cdot) $ is given as
\begin{equation*}
V(x_\Ti) = x_\Ti^\top P_c x_\Ti
\end{equation*}
with $ P_c \in \mb R^{2 \times 2} $, while in the distributed MPC formulation
\begin{equation*}
V(x_\Ti) = x_{1,\Ti}^\top P_{1,d} x_{1,\Ti} + x_{2,\Ti}^\top P_{2,d} x_{2,\Ti} = x_\Ti^\top P_d x_\Ti
\end{equation*}
where $ P_{1,d}, P_{2,d} \in \mb R $. The matrices $ P_c $ and $ P_d $ are computed by solving LMIs derived by the stability condition in Theorem \ref{thm::ST} and are given as,
\begin{equation*}
P_c = \begin{bmatrix}
3.46 & 0.13 \\
0.13 & 1.25
\end{bmatrix} \text{ and } P_d = \begin{bmatrix}
8.07 & 0 \\
0 & 4.25
\end{bmatrix}.
\end{equation*}
This LMI also provides the controller $ K_c $ while $ K_d $ is evaluated jointly together with the value function and the terminal sets in each iteration using the techniques discussed in Section \ref{sec::design}. Given $ K_c $, the maximum invariant terminal set, $ \mc X_\infty $, is computed using routines developed in MPT 3.0 toolbox \cite{MPT3}, while the ellipsoidal invariant terminal set, $ \mc X_\Fi $, is computed solving a linear optimization problem, described in \cite[\S 5.2]{Boyd1994}. We refer to the MPC Problem \eqref{Centralized} formulated with $ \mc X_\infty $ as ($ C $-Max.) and $ \mc X_\Fi $ as ($ C $-Ellip.). In the following we compare these centralized invariant terminal sets to the proposed decentralized adaptive one, given as $ \wt{\mc X}_f(\alpha_1, \alpha_2) = \{(x_1, x_2) \in \mb R^2: x_{1}^\top P_{1,d} x_{1} \le \alpha_1,\; x_{2}^\top P_{2,d} x_{2} \le \alpha_2 \} $ where $ \alpha_1, \alpha_2 $ are positive scalars computed with the methods presented in Section \ref{sec::design}. We refer to the DMPC Problem \eqref{DecentralizedAdaptive} formulated with $ \wt{\mc X}_\Fi(\alpha_1, \alpha_2) $ as ($ D $-Adap.). 

\begin{figure}[t]
	\centering
	\begin{minipage}{0.49\textwidth}
		\subfigure[]{\includegraphics[width = \textwidth]{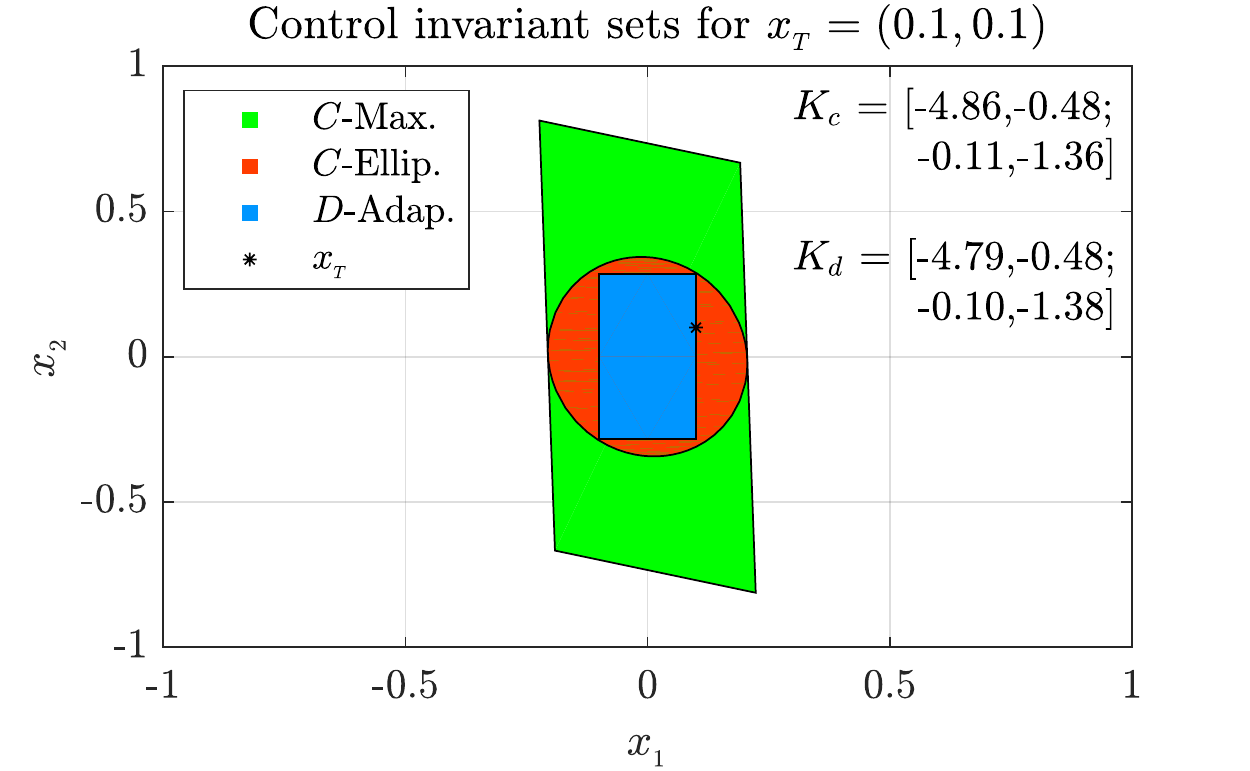}}
	\end{minipage}
	\begin{minipage}{0.49\textwidth}
		\subfigure[]{\includegraphics[width = \textwidth]{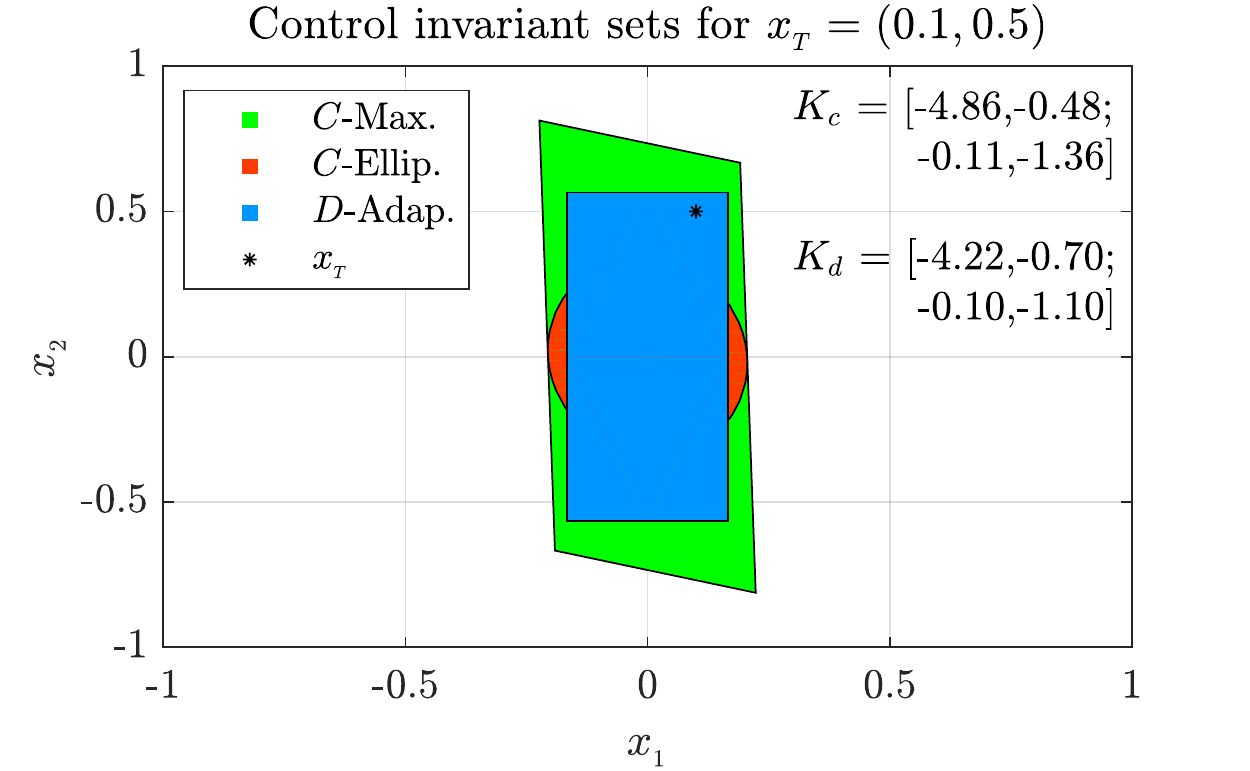}}
	\end{minipage}
	\caption{Shapes of $ \mc X_\infty $ (green), $ \mc X_\Fi $ (red) and $ \wt{\mc X}_\Fi(\alpha_1, \alpha_2) $ (green) for different terminal states $ x_\Ti $.}
	\label{fig::TE_Sets}
\end{figure}
In Fig. \ref{fig::TE_Sets}, the shapes of $ \mc X_\infty $ (green), $ \mc X_\Fi $ (red) and $ \wt{\mc X}_\Fi(\alpha_1, \alpha_2) $ (green) are depicted for different terminal states $ x_\Ti $. Note that if $ x_\Ti $ lies outside the terminal invariant set then the respective MPC problem is infeasible. That being said, this example highlights the ability of $ \wt{\mc X}_\Fi(\alpha_1, \alpha_2) $ to adapt so as to include terminal state $ x_\Ti $ in its interior; hence, appropriately adapting the feasibility domain of Problem ($ D $-Adap.). This adaptation is achieved by adjusting the values of the terminal controller $ K_d $, as it is reported in Fig. \ref{fig::TE_Sets}.

\begin{figure}[t]
	\begin{minipage}{0.33\textwidth}
		\subfigure[]{\includegraphics[width = \textwidth]{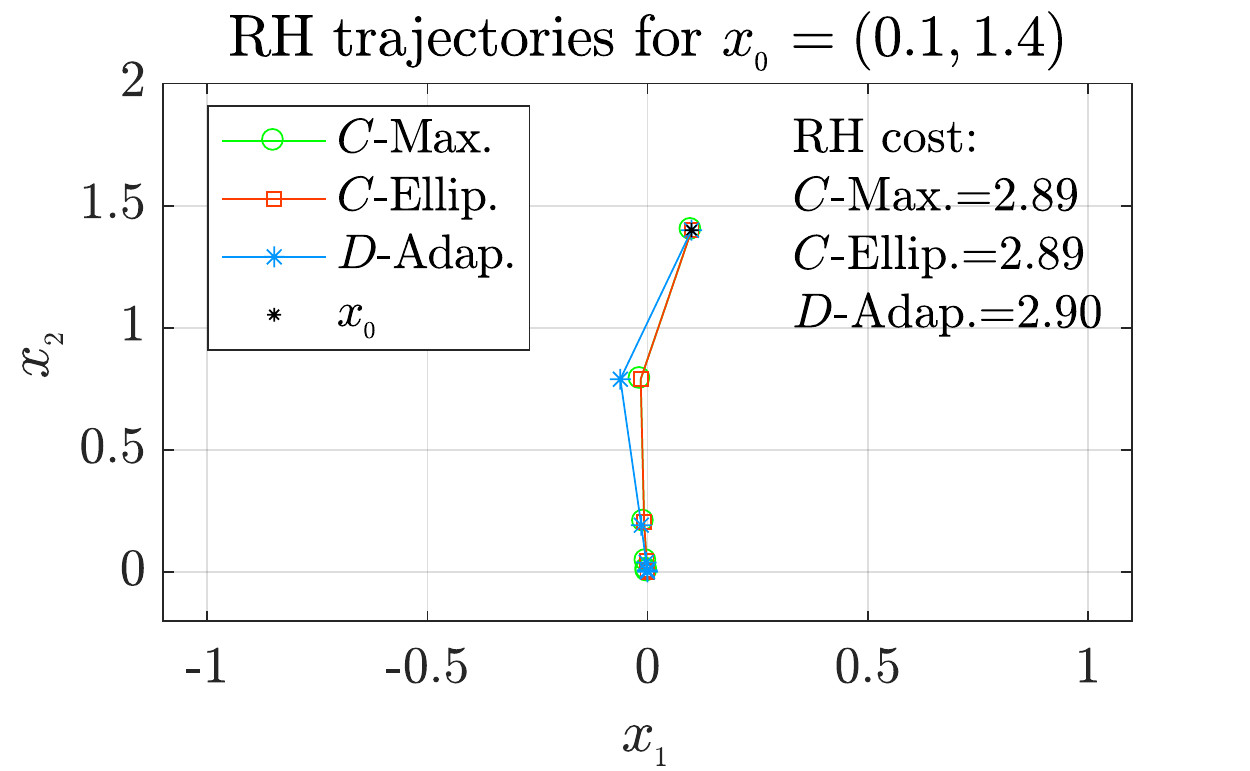}}
	\end{minipage}\hfill
	\begin{minipage}{0.33\textwidth}
		\subfigure[]{\includegraphics[width = \textwidth]{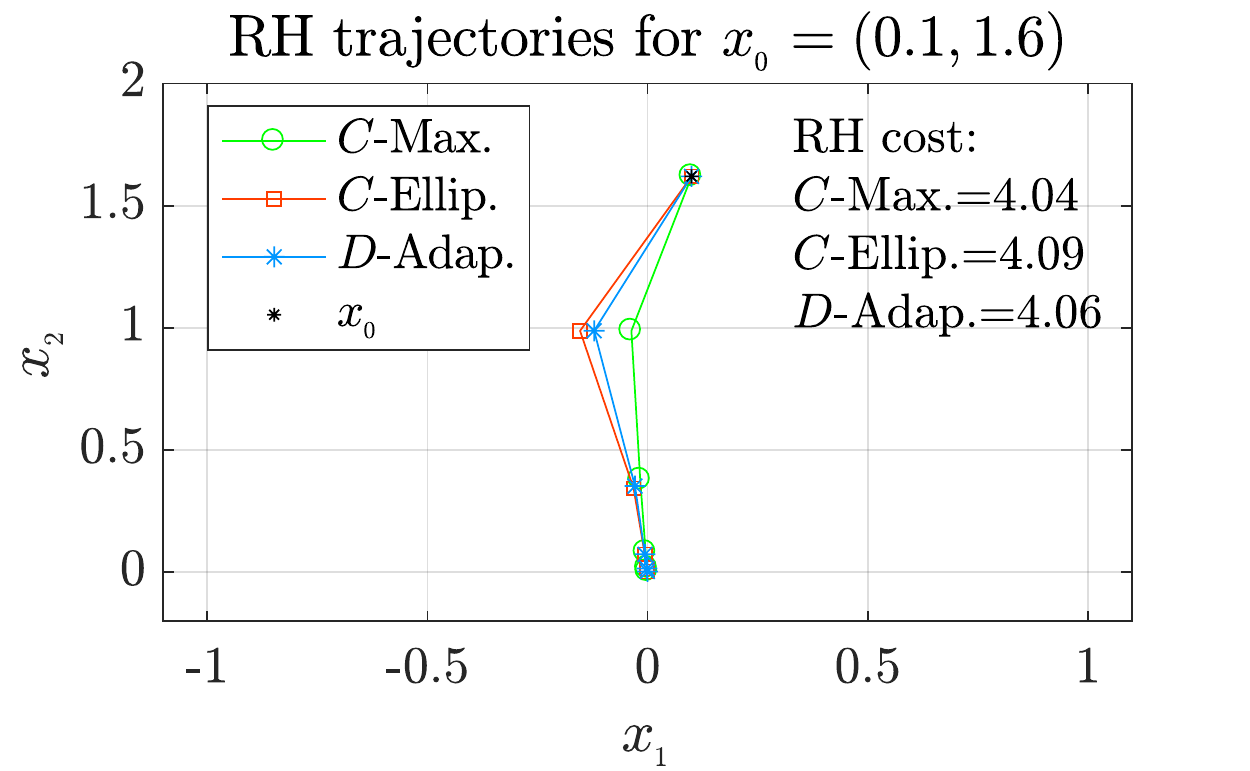}}
	\end{minipage}\hfill
	\begin{minipage}{0.33\textwidth}
		\subfigure[]{\includegraphics[width = \textwidth]{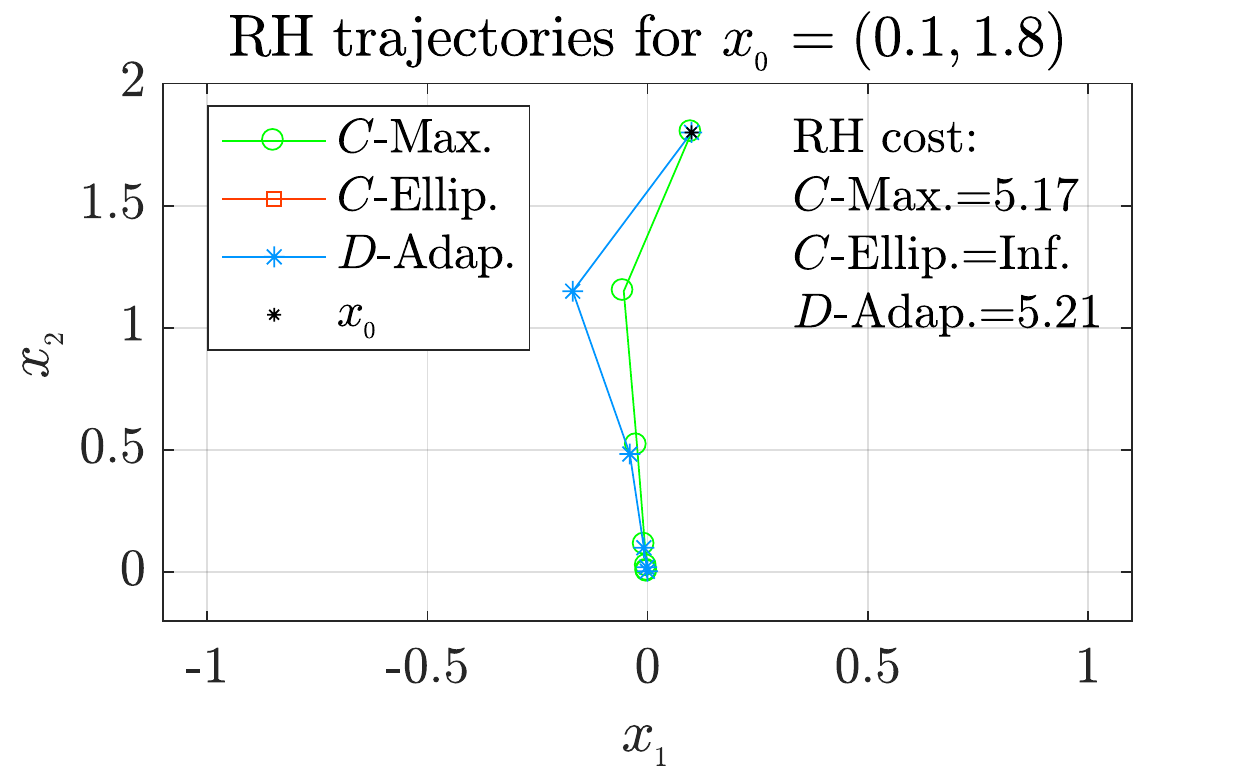}}
	\end{minipage}
	\caption{Closed-loop receding horizon performance of the system for different initial states $ x_0 $ and horizon $ T = 2 $.}
	\label{fig::TE_RH}
\end{figure}
The closed-loop behavior of the system for different initial conditions is now investigated. We choose a time horizon $ T = 2 $ and evaluate the performance of the system on a receding horizon implementation, i.e., the first input resulting from the respective MPC ($ C $-Max.), ($ C $-Ellip.)  and ($ D $-Adap.) optimization problem is applied to the system dynamics, and the next state is evaluated. As a metric the cost of operating the system until convergence to the origin is used. We report this comparison results in Fig. \ref{fig::TE_RH}. It can be observed that if the initial state $ x_0 $ is close to the origin then the MPC Problems ($ C $-Max.) and ($ C $-Ellip.) achieve the exact same cost although their size differs. This is not surprising since $ K_c $ is the optimal controller for the infinite horizon problem, thus, if the terminal state of Problem ($ C $-Max.) in the first iteration is the same to Problem ($ C $-Ellip.), then the two lead to the same solution. However, as the initial state is chosen further away from the origin the effort the system needs to achieve a feasible solution is high for ellipsoidal invariant terminal sets. In this case, the performance of the system under the DMPC controller ($ D $-Adap.) is better from the one achieved by the MPC controller ($ C $-Ellip.). The importance of considering invariant terminal sets that can adapt on the initial system state is highlighted when observing in Fig.~\ref{fig::TE_RH} that the DMPC controller ($ D $-Adap.) is feasible for initial states for which a centralized solution with ellipsoidal invariant terminal sets does not exist. This is attributed to the methods ability to modify the size of its terminal region by appropriately adapting its terminal controller $ K_d $ while satisfying the stability and invariance conditions.

\subsection{Spring-mass-dampers}
\begin{figure}[ht]
	\centering
	\input{springMassDamper.tex}
	\caption{A chain of four masses connected by springs and dampers.}
	\label{fig::Sim}
\end{figure}
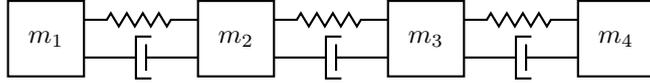
We now consider a series of masses that are connected by springs and dampers and arranged in a chain formation, exemplified in Fig. \ref{fig::Sim}. The values of the masses, spring constants and damping coefficients are chosen uniformly at random from the intervals $ [5,\,10] $kg, $ [0.8,\,1.2] $N/m and $ [0.8,\,1.2] $Ns/m, respectively. We assume that each $ i $-th mass is an individual system with its state vector $ x_{i,t} \in \mb R^2 $ representing the position and velocity deviation from the system's equilibrium state, and its input $ u_{i,t} \in \mb R $ denoting the force applied to the $ i $-th mass. We assume that the states and inputs are constrained as
\begin{equation*}
[-2, -5]^\top \le x_t \le [2, 5]^\top \text{ and } -u_c \le u_t \le u_c,
\end{equation*}
where $ u_c $ is chosen uniformly at random from the interval $ [2,\,4] $N. The masses are initially at rest and positioned uniformly at random within the intervals $ [-2,\,-1.8] $m and $ [1.8,\,2] $m.

The continuous-time dynamics of this interconnected dynamical system naturally admits a distributed structure. The prediction control model is obtained by the discretization of the system's continuous dynamics using forward Euler with the sampling time $ 0.1 $s. Although inexact, Euler discretization is chosen as to preserve the distributed structure of the system. On the contrary, the discrete-time simulation model of the system is obtained using the exact zero-order hold discretization method with the sampling time $ 0.1 $s. The objective function of each system is of the quadratic form \eqref{eq::StC} with $ Q_{\Ni} = \text{diag}(1,1) $ and $ R_i= 0.1 $. We refer to the MPC Problem \eqref{Centralized} formulated with $ \mc X_\infty $ as ($ C $-Max.) and $ \mc X_\Fi $ as ($ C $-Ellip.). We refer to the proposed DMPC Problem \eqref{DecentralizedAdaptive} formulated with $ \wt{\mc X}_{\Fi}(\alpha)  $ as ($ D $-Adap.).

\begin{figure}[t]
	\centering
	\begin{minipage}{0.49\textwidth}
		\subfigure[]{\includegraphics[width = \textwidth]{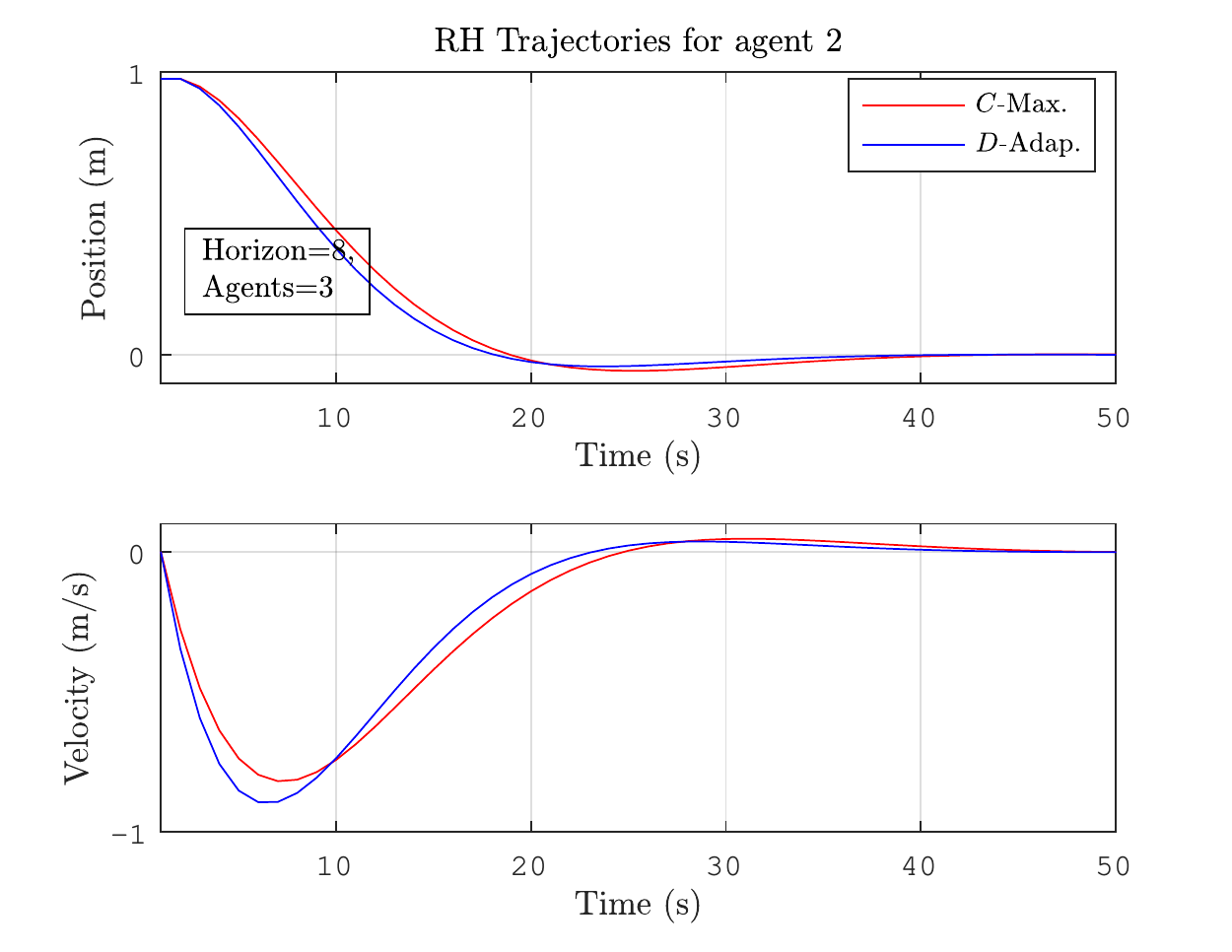} \label{fig::Trajecta}}
	\end{minipage}
	\begin{minipage}{0.49\textwidth}
		\subfigure[]{\includegraphics[width = \textwidth]{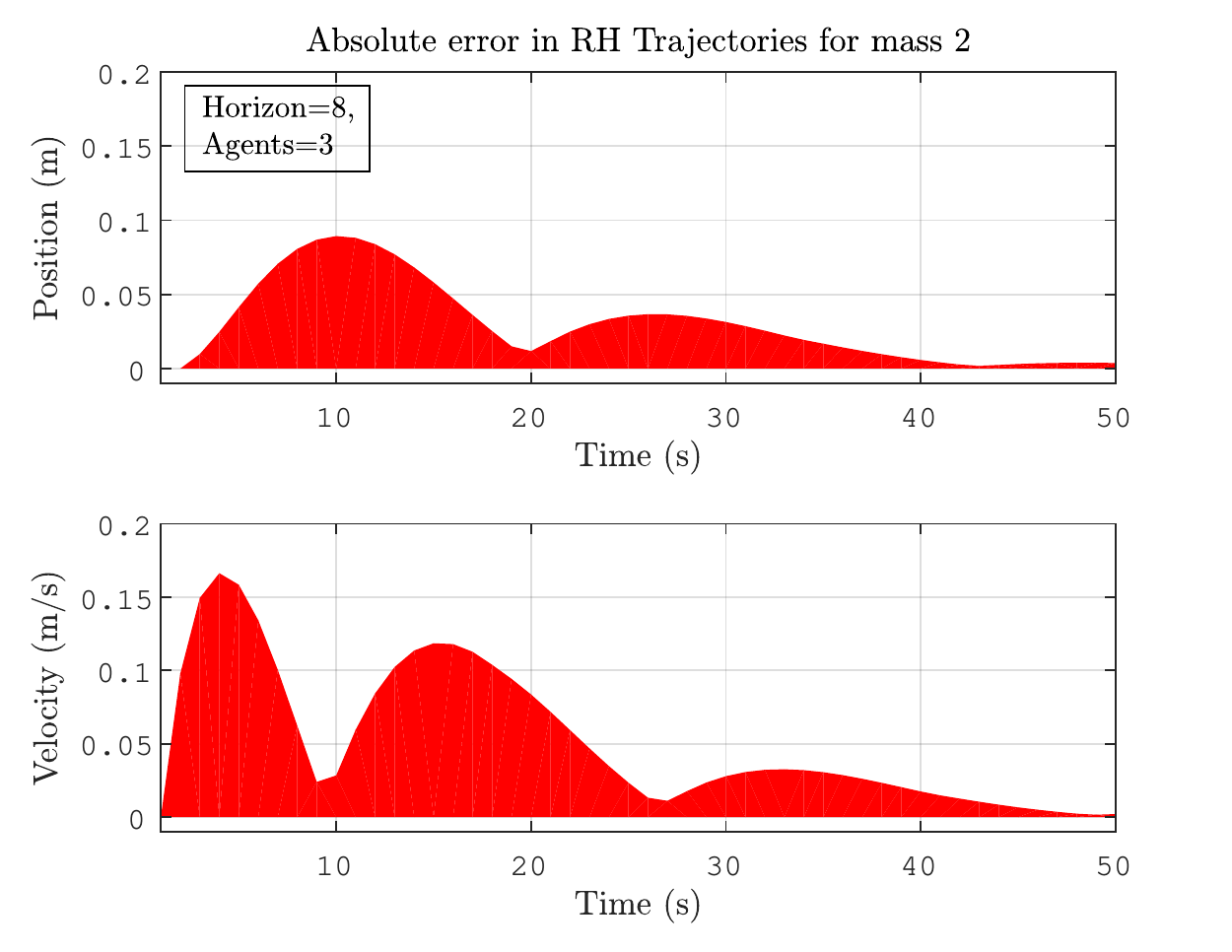} \label{fig::Trajectb}}
	\end{minipage}
	\caption{(a) Receding horizon (RH) position and velocity trajectories for centralized and decentralized methods and (b) absolute error of trajectories in time over 100 randomly generated and initialized systems.}
\end{figure}
The performance of the system is evaluated on a receding horizon implementation. We use as a metric the cost of operating the system until convergence to the system's equilibrium state. Initially, we conduct a closed-loop simulation experiment for a system comprising three masses and a prediction horizon of $ T = 8 $.  In Fig. \ref{fig::Trajecta}, the trajectories generated for the MPC ($ C $-Max.) and DMPC ($ D $-Adap.) problems are shown. We observe that these trajectories are very similar, which illustrates the proximity in performance between the centralized and distributed designs. To better quantify the error between these approaches over time, we repeated these simulations for 100 randomly generated systems of the same dimension and we report the results in \ref{fig::Trajectb}. We observe that the error remains bounded over time and eventually decreases to zero as the system converges to the origin. This verifies the proximity in performance between the proposed DMPC ($ D $-Adap.) and the well-established centralized MPC ($ C $-Max.) approach.

\begin{figure}[t]
	\centering
	\includegraphics[width = 0.5\textwidth]{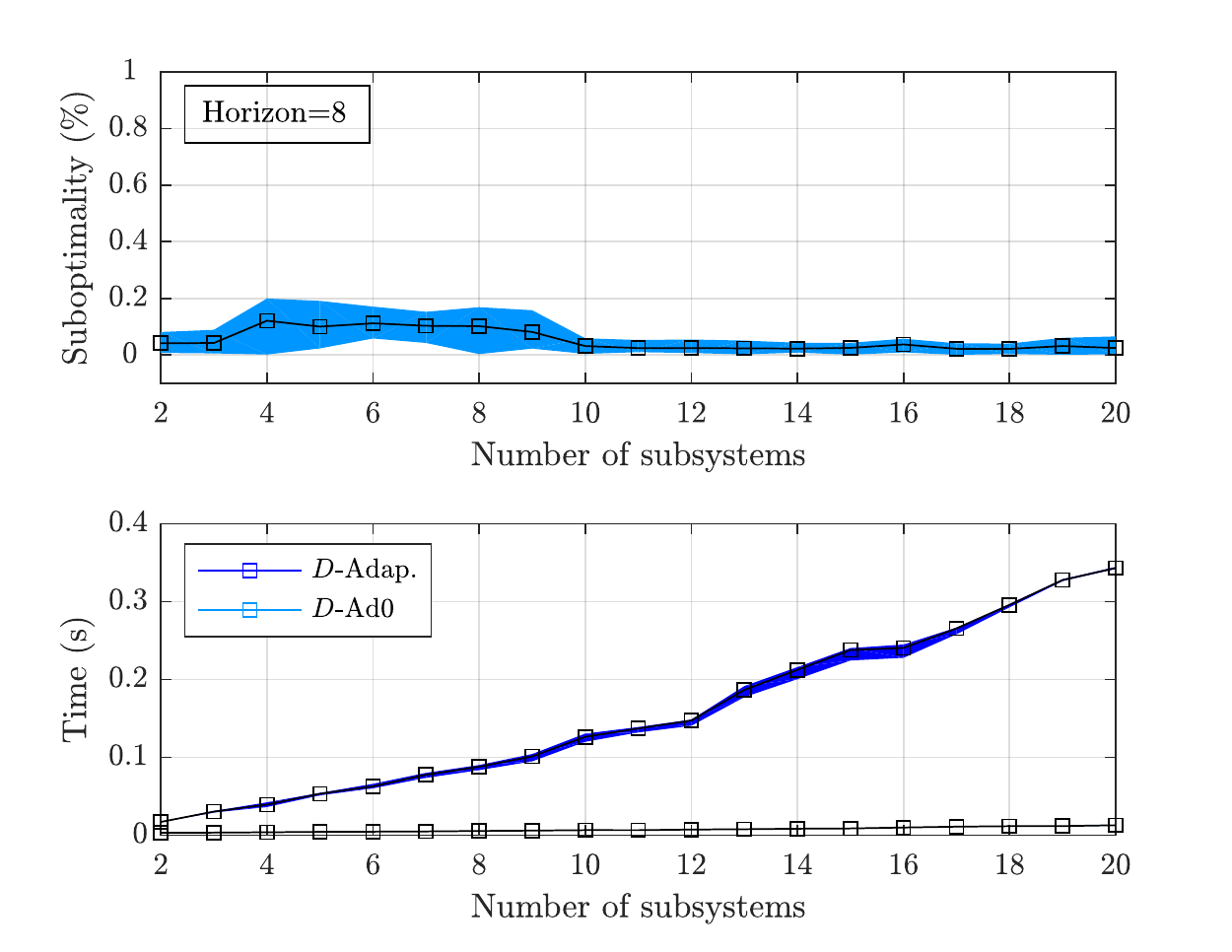}
	\caption{Comparison of decentralized adaptive methods where $ D $-Ad0 only assumes adaptation of the terminal sets once at time $ t=0 $.}
	\label{fig::SMD_AdapComp}
\end{figure}
However, the proposed approach relies on the adaptation of the invariant terminal sets in each receding horizon simulation which involves the formulation and solution of a semi-definite program. To avoid the computational burden associated with it, we compare the proposed fully adaptive method ($ D $-Adap.) with its simplification ($ D $-Ad0) in which the adaptation of the invariant terminal sets is only performed once at time $ t = 0 $ to account for the effect of the initial state of the system. Then, we enforce these computed terminal sets for the rest of the receding horizon simulations. This way we resort to solving a quadratically constraints quadratic program instead of a semidefinite one. These two approaches are compared as the number of masses in the system increases where for each topology we generate 100 random system instances. We use as metrics for this comparison the mean solution time and the cost of the receding horizon simulations until convergence to the origin is achieved. The results are reported in Fig. \ref{fig::SMD_AdapComp}. It is observed that adapting the invariant terminal sets in every iteration of the receding horizon simulation provides a slightly better solution quality with respect the case where the adaptation is only performed once at time $ t = 0 $. This can be attributed to the dissipative nature of the spring-mass-damper system for which the initial displacement is the determining factor for the shape of the invariant terminal sets. On the other hand, the computational benefit occurring when using the $ D $-Ad0 method is considerable since the simplified approach only requires a fraction of the time to generate the solution of the problem.

\begin{figure}[t]
	\begin{minipage}{0.5\textwidth}
		\subfigure[]{\includegraphics[width = \textwidth]{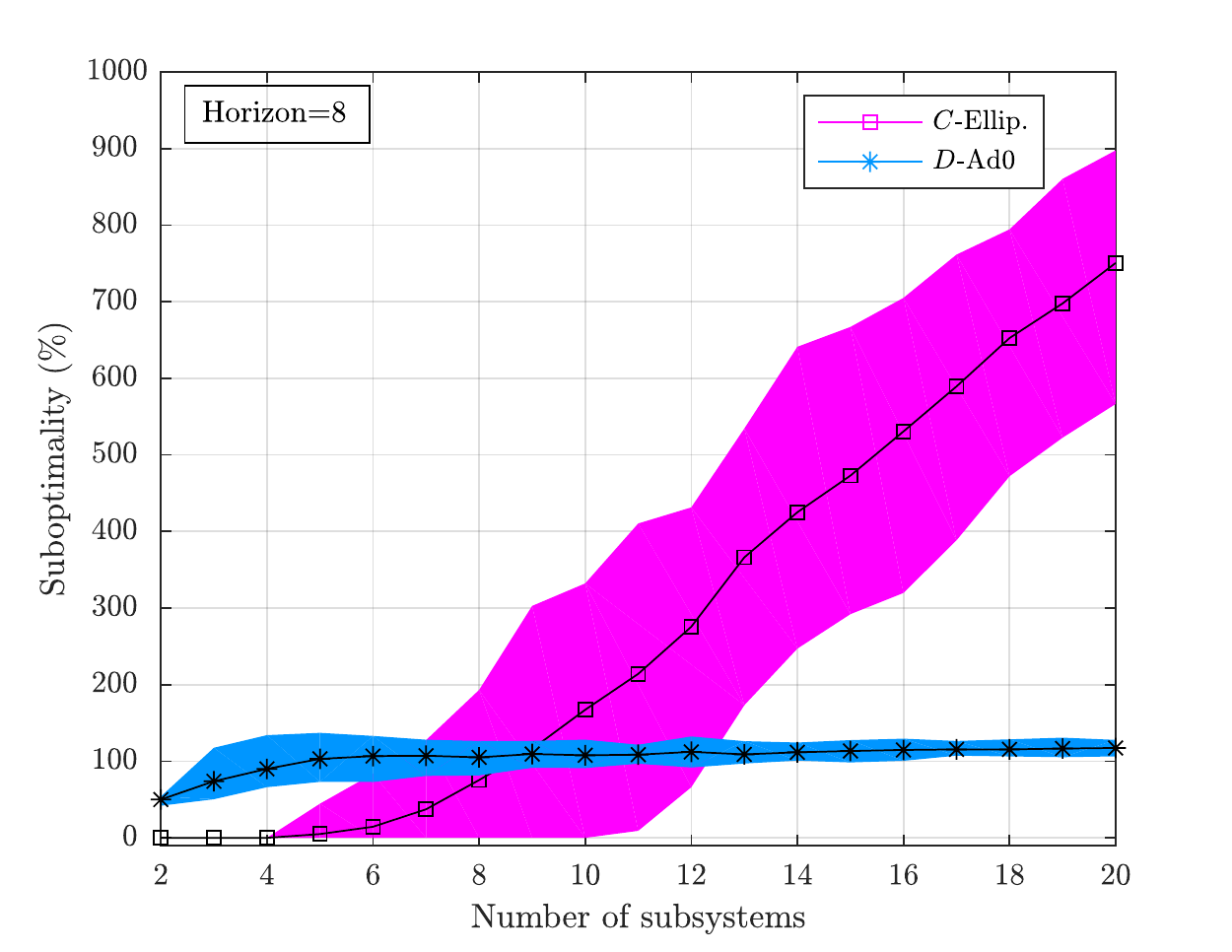}\label{fig::SMD_MehtodsCompa}}
	\end{minipage}\hfill
	\begin{minipage}{0.5\textwidth}
		\subfigure[]{\includegraphics[width = \textwidth]{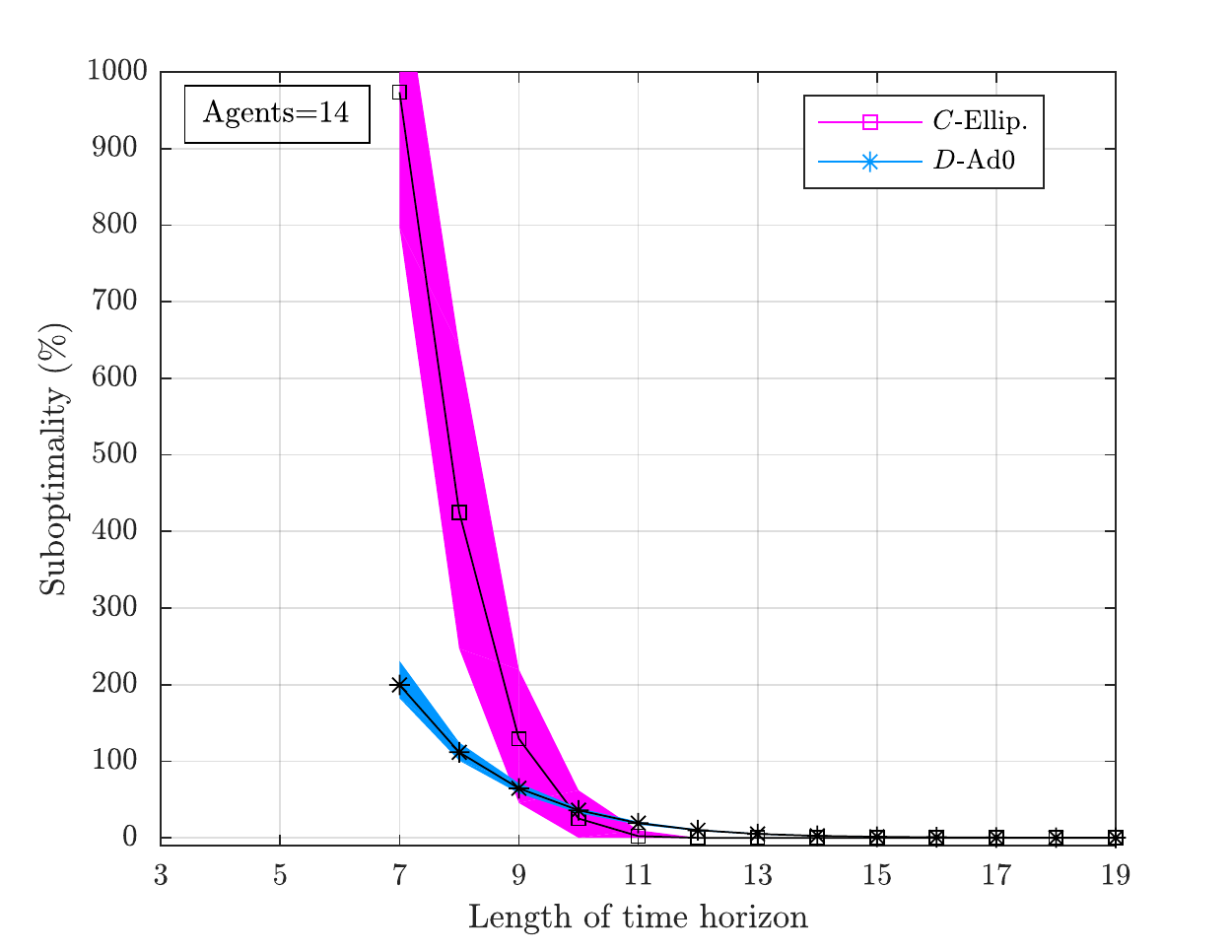}\label{fig::SMD_MehtodsCompb}}
	\end{minipage}
	\caption{Suboptimality of centralized MPC method with ellipsoidal invariant terminal sets and distributed MPC method with adaptive invariant terminal sets as (a) the number of masses in the system increases and (b) the time horizon of the MPC formulation increases. }
	
\end{figure}
To better quantify the performance comparison between the proposed adaptive DMPC approach and centralized MPC designs, we conducted several simulation experiments for systems with different horizons and number of masses. The comparison is performed on the suboptimality of the respective methods during receding horizon simulations using as basis the cost associated with the centralized MPC Problem ($ C $-Max.). Since we use the $ D $-Ad0 method as the proposed methodology for comparison, the solution times are comparable and are not reported. We remark though that the proposed decentralized method achieves a faster solution time due to their distributed structure which can potentially be further exploited by dedicated distributed computation algorithms. Fig \ref{fig::SMD_MehtodsCompa} shows the cost associated as the number of subsystems in the system increases, where the simulation horizon is kept constant as $ T = 8 $. We observe that the proposed method considerably outperforms even the centralized approach with ellipsoidal terminal sets as the number of subsystems increases. This is attributed to the proposed method's ability to adapt on the initial condition. As expected, the suboptimality gap increases with respect to the number of masses in the system. We note, however, that in all instances this suboptimality gap is fairly small, which indicates the efficiency of the proposed distributed design method. Finally, Fig. \ref{fig::SMD_MehtodsCompb} shows the cost associated with the length of prediction horizon for a system comprising five masses. We observe that the increase of the horizon length results in cost convergence for the compared methods. Notably as the horizon increases the centralized methods outperform the proposed decentralized one since large horizons make the use of terminal sets and value functions obsolete with the system being capable of steering its states close to the equilibrium state within the considered prediction horizon time.

\section{Conclusion} \label{sec::conclusion}
In this paper, we presented a design approach for distributed cooperative MPC that encapsulates the design of the distributed terminal controller, value function and invariant set in the MPC formulation. Conditions for Lyapunov stability and invariance are imposed in the design problem in a way that allows the value function and terminal invariant set to admit the desired distributed structure. This allows the resulting distributed MPC problem to be amendable to distributed computation algorithms. The proposed distributed MPC method couples the design of the terminal stabilizing controllers and invariant terminal sets with the current and predicted states of the system. The merits of considering adaptive invariant terminal sets is illustrated in a large-scale system that is composed of masses connected by springs and dampers. The closed-loop performance of the proposed distributed MPC approach is shown to outperform even the centralized MPC problem formulated with the ellipsoidal invariant terminal set for short prediction horizons.

Future work involves the extension of the proposed methodology to plug-and-play applications where the conditions for Lyapunov stability and terminal set invariance need to be evaluated in a completely distributed way. An important feature to be exploited is that only the new and a few of the existing distributed controllers may need to be redesigned.

\bibliographystyle{unsrt}
\bibliography{darivianos_abrv,Papers}

\end{document}

%% file: springMassDamper.tex
\begin{tikzpicture}
[main node/.style={draw,outer sep=0pt,thick}]

\tikzstyle{spring}=[thick,decorate,decoration={zigzag,pre length=0.3cm,post length=0.3cm,segment length=6}]
\tikzstyle{damper}=[thick,decoration={markings,  
  mark connection node=dmp,
  mark=at position 0.5 with 
  {
    \node (dmp) [thick,inner sep=0pt,transform shape,rotate=-90,minimum width=15pt,minimum height=3pt,draw=none] {};
    \draw [thick] ($(dmp.north east)+(2pt,0)$) -- (dmp.south east) -- (dmp.south west) -- ($(dmp.north west)+(2pt,0)$);
    \draw [thick] ($(dmp.north)+(0,-5pt)$) -- ($(dmp.north)+(0,5pt)$);
  }
}, decorate]

\node[main node] (M1) [minimum width=1cm,minimum height=1cm,xshift=-2.5cm] {$m_1$};
\node[main node] (M2) [minimum width=1cm,minimum height=1cm] {$m_2$};
\node[main node] (M3) [minimum width=1cm,minimum height=1cm,xshift=2.5cm] {$m_3$};
\node[main node] (M4) [minimum width=1cm,minimum height=1cm,xshift=5cm] {$m_4$};


\draw [spring] ($(M1.north east)!0.25!(M1.south east)$) -- ($(M2.north west)!0.25!(M2.south west)$);
\draw [damper] ($(M1.north east)!0.75!(M1.south east)$) -- ($(M2.north west)!0.75!(M2.south west)$);

\draw [spring] ($(M2.north east)!0.25!(M2.south east)$) -- ($(M3.north west)!0.25!(M3.south west)$);
\draw [damper] ($(M2.north east)!0.75!(M2.south east)$) -- ($(M3.north west)!0.75!(M3.south west)$);

\draw [spring] ($(M3.north east)!0.25!(M3.south east)$) -- ($(M4.north west)!0.25!(M4.south west)$);
\draw [damper] ($(M3.north east)!0.75!(M3.south east)$) -- ($(M4.north west)!0.75!(M4.south west)$);

\end{tikzpicture}